\documentclass[draft]{amsart}
\usepackage{amsfonts,amssymb,latexsym,amsmath, amsxtra,url, mathrsfs}
\usepackage[utf8]{inputenc}

\usepackage{graphicx}
\graphicspath{ {./Figures/} }

\usepackage{rotating}

\usepackage{tikz}
\usepackage{caption}

\usepackage{hyperref}
\hypersetup{colorlinks=true,linkcolor=blue }

\theoremstyle{plain}
\newtheorem{theorem}{Theorem}[section]
\newtheorem*{theorem*}{Theorem}
\newtheorem{corollary}[theorem]{Corollary}

\newtheorem{lemma}[theorem]{Lemma}
\newtheorem{proposition}[theorem]{Proposition}
\newtheorem*{conjecture*}{Conjecture}

\theoremstyle{definition}
\newtheorem{definition}[theorem]{Definition}
\newtheorem{remark}[theorem]{Remark}
\theoremstyle{remark}

\numberwithin{equation}{section}

\newcommand{\SCP}{\mathrm{SCP}}

\title{Weighted Cylindric Partitions}

\author{Walter Bridges }
\address{[W.B.] University of Cologne, Department of Mathematics and Computer Science, Weyertal 86-90, 50931 Cologne, Germany}
\email{wbridges@uni-koeln.de}

\author[Ali K. Uncu]{Ali K. Uncu}
\address{[A.K.U.] Johann Radon Institute for Computational and Applied Mathematics, Austrian Academy of Science, Altenbergerstraße 69, A-4040 Linz, Austria}
\email{akuncu@ricam.oeaw.ac.at}
\address{[A.K.U.] University of Bath, Faculty of Science, Department of Computer Science, Bath, BA2 7AY, UK}
\email{aku21@bath.ac.uk}

\thanks{Research of the first author is partially supported by the DFG (Projektnummer 281071066 TRR 191). Research of the second author is partly supported by EPSRC grant number EP/T015713/1 and partly by FWF grant P-34501N}

\keywords{Cylindric partitions, skew double shifted plane partitions, partition diamonds, partition identities, weighted partition identities}
  
\subjclass[2010]{Primary 05A15; Secondary 05A17, 11B65, 11P81, 11P84}

\date{\today}

\begin{document}

\maketitle

\begin{abstract}
    Recently Corteel and Welsh outlined a technique for finding new sum-product identities by using functional relations between generating functions for cylindric partitions and a theorem of Borodin. Here, we extend this framework to include very general product-sides coming from work of Han and Xiong. In doing so, we are led to consider structures such as weighted cylindric partitions, symmetric cylindric partitions and weighted skew double shifted plane partitions. We prove some new identities and obtain new proofs of known identities, including the G\"{o}llnitz--Gordon and Little G\"{o}llnitz identities as well as some beautiful Schmidt-type identities of Andrews and Paule.
\end{abstract}

\section{Introduction}

Sum-product identities lie at the heart of the theory of integer partitions and $q$-series.  For example, the two Rogers--Ramanujan identities state that for $\varepsilon \in \{0,1\}$, we have
\begin{equation}\label{E:RR}
\sum_{n \geq 0} \frac{q^{n^2+\varepsilon n}}{(q;q)_n} = \frac{1}{(q^{1+\varepsilon},q^{4-\varepsilon};q^5)_{\infty}},
\end{equation}
where throughout the article we use the $q$-Pochhammer notation, $(a;q)_n:= \prod_{j=1}^n(1-aq^j)$ for $n \geq 1$ with $(a;q)_0:=1 $ and $(q;q)_n^{-1}=0$ for $n \leq -1$.  The Rogers--Ramanujan identities lie deep enough within mathematics to have quite a variety of proofs---combinatorial \cite{GM}, so-called ``motivated proofs'' (via recurrences) \cite{AB}, and via affine Lie algebras \cite{LW}, to name just a few.  We refer the reader to Sills' recent book \cite{Sills} for a thorough history.  The Rogers--Ramanujan identities can be interpreted in terms of integer partitions, and such identities in which a $q$-hypergeometric series equals a product as simple as the right-hand side of \eqref{E:RR} are rare.  Thus, in addition to discovering new identities, it is worthwhile to provide new connections to combinatorial structures beyond integer partitions.

Corteel and Welsh \cite{CW} recently showed how cylindric partitions can be used in the discovery of new sum-product identities.  (We give precise definitions of these objects and some generalizations in Section 2.)  Cylindric partition generating functions for any profile were shown by Borodin \cite{B} to equal infinite products, and ``sum-sides'' were then found by Corteel and Welsh to satisfy $q$-difference equations, which they showed occur naturally for cylindric partitions.  This led to alternative proofs of all four of Andrews--Schilling--Warnnar's $A_2$~Rogers--Ramanujan identities \cite{ACW}, as well as a proof of a new fifth identity which was conjectured in \cite{FW}. One such identity Corteel--Welsh proved in \cite[Theorem 1.1]{CW} is as follows: 

\begin{equation}\label{E:CorteelWelshwidth7rank3}
    \sum_{n_1, n_2 \geq 0} \frac{q^{n_1^2+n_2^2-n_1n_2+n_1+n_2}}{(q;q)_{n_1}}\begin{bmatrix} 2n_1 \\ n_2 \end{bmatrix}_q= \frac{1}{(q^2,q^3,q^3,q^4,q^4,q^5;q^7)_{\infty}}.
\end{equation}
Here the Gaussian polynomials are defined as \[\begin{bmatrix} n \\ m \end{bmatrix}_q=\left\{\begin{array}{cl}\displaystyle\frac{(q;q)_n}{(q;q)_m(q;q)_{n-m}}, &\text{if }n\geq m\geq 0, \\[-1.5ex]\\ 0, & \text{otherwise.}\end{array}\right.\] Recent work of Corteel, Dousse and the second author \cite{CDU}, as well as Warnaar \cite{W}, uses Corteel--Welsh's machinery to prove and conjecture many other sum-product identities.

In the present article, we extend the search for sum-product identities to certain generalizations of cylindric partitions, recently proved to also have infinite product generating functions by Han and Xiong \cite{HX}.  Essentially, one can write {\it any} product $(q^{b_1},\dots, q^{b_{r}};q^{b_r})^{-1}_{\infty}$ as a generating function for some weighted cylindric partitions (see Remark \ref{R:differentproducts}).  By then extending Corteel and Welsh's machinery to this more general setting, we were able to prove the following identities.  

\begin{theorem}\label{T:scpwidth2}  For $z \in \mathbb{C}$ and $|q|<1$, we have
\[\sum_{n\geq 0}(-1)^{n} q^{4n^2} \frac{(q^2,-q^4;q^4)_{n}}{(q^4;q^4)_{2n}} \left(1 - \frac{q^{4n+1}z}{1+q^{4n+2}} \right) z^{2n} = (zq, -zq^3;q^4)_\infty.\]
\end{theorem}

We offer a combinatorial proof of a specialization of Theorem \ref{T:scpwidth2} in Section \ref{S:proofs}.  Next, we have the following family of four identities.

\begin{theorem}\label{T:scpwidth3}
For $|q|<1$, we have  
\begin{align}
    &\sum_{n,m \geq 0}  (-1)^mq^{3\binom{n+1}{2}-3m(m+1)}\frac{(-q,-q^5;q^6)_m}{(q^6;q^6)_m(q^3;q^3)_{n-2m}}=\frac{(q^4,q^8;q^{12})_{\infty}}{(q^6;q^{12})_{\infty}}, \label{E:scpCaseC} \\
    &\sum_{n,m \geq 0}  (-1)^{m+1} q^{3{n-1\choose 2}+2n -3m(m+1)-1} \frac{(-q,-q^5;q^6)_m}{(q^6;q^6)_m(q^3;q^3)_{n-2m}} (1-q^{3n+1}+q^{3n-6m})=\frac{(q;q^6)_{\infty}(q^{10};q^{12})_{\infty}}{(q^5;q^{6})_{\infty}}, \label{E:scpCaseB} \\
    &\sum_{n,m \geq 0} (-1)^{m+1} q^{3{n+1\choose 2} -3m(m+1)-1} \frac{(-q,-q^5;q^6)_m}{(q^6;q^6)_m(q^3;q^3)_{n-2m}} (1-q^{3n+1}+q^{3n-6m})=\frac{(q^2,q^{10};q^{12})_{\infty}}{(q^6;q^{12})_{\infty}}, \label{E:scpCaseA} \\
    \label{E:scpCaseG}
    &\sum_{n,m \geq 0}  (-1)^{m+1} q^{3{n+1\choose 2}-2n-3m(m+1)-1} \frac{(-q,-q^5;q^6)_m}{(q^6;q^6)_m(q^3;q^3)_{n-2m}}\\\nonumber &\hspace{4cm}\times(1 + q^{3n-12m-3} (1 + q^{1 + 6 m}) (q^{3 n} - 
    q^{6 m} (1 + q^3))=\frac{(q^2,q^5,q^{11};q^{12})_{\infty}}{(q;q^6)_{\infty}}. 
\end{align}
\end{theorem}

We further use our machinery to offer new proofs of the Little G\"{o}llnitz and G\"{o}llnitz--Gordon identities \cite[Theorem 7.11]{A}, and \cite{G}.
\begin{theorem}[Little G\"{o}llnitz and G\"{o}llnitz--Gordon Identities]\label{C:gollnitz}
We have
\begin{align}
&\sum_{n \geq 0} \frac{\left(-q;q^2\right)_n}{\left(q^2;q^2\right)_n}q^{n^2+2n} =\frac{1}{\left(q^3,q^4,q^5;q^8\right)_{\infty}}, \label{E:GoellnitzGordon1} \\
&\sum_{n \geq 0} \frac{\left(-q;q^2\right)_n}{\left(q^2;q^2\right)_n}q^{n^2+n} =\frac{1}{\left(q^3;q^4\right)_{\infty}(q^2;q^8)_{\infty}}, \label{E:LittleGoellnitz1}\\
&\sum_{n \geq 0} \frac{\left(-q;q^2\right)_n}{\left(q^2;q^2\right)_n}q^{n^2} =\frac{1}{\left(q,q^4,q^7;q^8\right)_{\infty}}, \label{E:GoellnitzGordon2} \\
&\sum_{n \geq 0} \frac{\left(-q^{-1};q^2\right)_{n}}{\left(q^2;q^2\right)_n}q^{n^2+n} =\frac{1}{\left(q;q^4\right)_{\infty}\left(q^6;q^8\right)_{\infty}} \label{E:LittleGoellnitz2}.
\end{align}
\end{theorem}

As a final application, we give refinements of the following Schmidt-type partition identities appearing in recent work of Andrews and Paule \cite{AP}.  Identity \eqref{E:schmidtdistinct} was first conjectured in \cite{S} by Schmidt, who himself subsequently gave a proof.  The first two identities also independently appeared in work by the second author \cite{U} in the context of certain weighted partition identities. Interested readers are invited to look into these works.

A {\it partition} $\lambda$ of an integer $n$ is a sequence of positive integers whose {\it parts} satisfy
$$
\lambda_1 \geq \lambda_2 \geq \dots \geq \lambda_{\ell} \qquad \text{and} \qquad |\lambda|:=\sum_{j=1}^{\ell} \lambda_j=n.
$$
We call such $\lambda$ a partition of \textit{size} $n$.

Let $\mathcal{P}$ (resp. $\mathcal{D}$) denote the set of partitions (resp. partitions with distinct parts---when the inequalities between the parts are all strict). Also let $\diamondsuit$ denote the set of {\it partition diamonds}, i.e. finite sequences of integers $\lambda_1, \lambda_2, \dots,$ where 
$$
\lambda_1 \geq  \begin{matrix}\lambda_2 \\[-2ex]\\ \lambda_3 \end{matrix} \geq \lambda_4 \geq  \begin{matrix}\lambda_5 \\[-2ex]\\ \lambda_6 \end{matrix}  \geq \lambda_7 \geq  \begin{matrix}\lambda_8 \\[-2ex]\\ \lambda_9 \end{matrix}  \geq \lambda_{10} \geq  \begin{matrix}\lambda_{11} \\[-2ex]\\ \lambda_{12} \end{matrix}  \geq \lambda_{13} \geq \dots.
$$
\begin{theorem}[Theorems 1, 2 and 4 of \cite{AP}, respectively]\label{T:AndrewsPaule}
We have
\begin{align}
&\sum_{\lambda \in \mathcal{D}}q^{\lambda_1+\lambda_3+\lambda_5+\dots}=\frac{1}{(q;q)_{\infty}}, \label{E:schmidtdistinct}  \\  &\sum_{\lambda \in \mathcal{P}}q^{\lambda_1+\lambda_3+\lambda_5+\dots}=\frac{1}{(q;q)_{\infty}^2},   \label{E:refinedschmidtunrestricted1q=1}  \\ &\sum_{\lambda \in \diamondsuit}q^{\lambda_1+\lambda_4+\lambda_7 + \dots} = \frac{(-q;q)_{\infty}}{(q;q)^3_{\infty}}. \label{E:schmidtdiamond} 
\end{align}
\end{theorem}

We prove the following two-variable versions.
\begin{theorem}\label{T:RefinedAndrewsPaule}
We have
\begin{align}
&\sum_{\lambda \in \mathcal{D}}z^{\lambda_1}q^{\lambda_1+\lambda_3+\lambda_5+\dots}=\sum_{n \geq 0} \frac{z^{2n}q^{n(n+1)}}{(zq;q)_n(zq;q)_{n+1}}, \label{E:refinedschmidtdistinct1} \\ &\sum_{\lambda \in \mathcal{D}}z^{\lambda_1}q^{\lambda_2+\lambda_4+\lambda_6+\dots}=1+\sum_{n \geq 1} \frac{z^{2n-1}q^{n(n-1)}}{(z,zq;q)_{n}}, \label{E:refinedschmidtdistinct2} \\  &\sum_{\lambda \in \mathcal{P}}z^{\lambda_1}q^{\lambda_1+\lambda_3+\lambda_5+\dots}=\frac{1}{(zq;q)_{\infty}^2}, \label{E:refinedschmidtunrestricted1} \\ &\sum_{\lambda \in \mathcal{P}}z^{\lambda_1}q^{\lambda_2+\lambda_4+\lambda_6+\dots}=\frac{1}{(1-z)(zq;q)_{\infty}^2},  \label{E:refinedschmidtunrestricted2} \\ &\sum_{\lambda \in \diamondsuit}z^{\lambda_1}q^{\lambda_1+\lambda_4+\lambda_7 + \dots} = \frac{(-zq;q)_{\infty}}{(zq;q)^3_{\infty}}. \label{E:refinedschmidtdiamond}
\end{align}
\end{theorem}
It is easy to observe that, as $z\mapsto 1$, \eqref{E:refinedschmidtunrestricted1} and \eqref{E:refinedschmidtdiamond} imply \eqref{E:refinedschmidtunrestricted1q=1} and \eqref{E:schmidtdiamond}, respectively.

  Identities \eqref{E:refinedschmidtdistinct1}--\eqref{E:refinedschmidtdiamond} have natural combinatorial interpretations; for example, \eqref{E:refinedschmidtdistinct1} implies the following.
\begin{corollary}\label{C:cor1}
The number of partitions into distinct parts with largest part $m$ and $n=\lambda_1+\lambda_3+\lambda_5+\dots$ equals the number of partitions of $n$ into unrestricted parts with largest hook length $m$.
\end{corollary}
Note that \eqref{E:schmidtdistinct} follows from Corollary~\ref{C:cor1}.   Similarly, setting $z \mapsto zq$ in \eqref{E:refinedschmidtdistinct2}, one has the following combinatorial interpretation.
\begin{corollary}
For $n \geq 1$, the number of partitions into distinct parts with largest part $m$ and $n=\lambda_1+(\lambda_2+\lambda_4+\lambda_6+\dots)$ equals the number of partitions of $n+1$ into unrestricted parts greater than 1 with largest hook length $m+1$.
\end{corollary}

In Section \ref{S:definition}, we define our objects of study---weighted cylindric partitions, symmetric cylindric partitions, and skew double shifted plane partitions---and state the product generating functions for these that follow from Han and Xiong's work in \cite{HX}.  In Section \ref{S:recurrences}, we record our Corteel--Welsh-type recurrences for two variable generating functions.  In Section \ref{S:proofs}, we use these recurrences to prove Theorems \ref{T:scpwidth2}, \ref{T:scpwidth3} and Theorem \ref{C:gollnitz}; Theorem \ref{T:scpwidth3} is proved with the aid of Mathematica.  In Section \ref{S:Applications}, we prove Theorem \ref{T:RefinedAndrewsPaule}.  We conclude in Section \ref{S:Outlook} by mentioning a few open problems and avenues for further work.  Appendix A contains a short proof of the modularity of the products in Proposition \ref{P:borodin}.  In Appendix B, we briefly mention the new computer algebra implementations developed in connection to this work.

\section*{Acknowledgements}

We would like to thank George E. Andrews, Jehanne Dousse, Guoniu Han, Ralf Hemmecke, and Peter Paule for their interest and encouragement. We would like to particularly thank Christian Krattenthaler and Wadim Zudilin for the discussion and their valuable feedback.

The first author is partially supported by the SFB/TRR 191 ``Symplectic Structures in Geometry, Algebra and Dynamics'', funded by the DFG (Projektnummer 281071066 TRR 191). Research of the second author is partly supported by EPSRC grant number EP/T015713/1 and partly by FWF grant P-34501N.

\section{Definitions and product sides}\label{S:definition}

Cylindric partitions were introduced by Gessel and Krattenthaler \cite{GK} as a type of repeating grid pattern of positive integers that are weakly decreasing along rows and columns; see Figure \ref{F:cylindric}.  Here, we will define them using the diagonal integer partitions.  For two partitions $\lambda$ and $\mu,$ we write $\lambda \succeq \mu$ if 
\begin{equation}\label{E:width1DSPP}\lambda_1 \geq \mu_1 \geq \lambda_2 \geq \mu_2 \geq \dots.\end{equation}
\begin{definition}
A {\it cylindric partition} of {\it width} $h$ with {\it profile} $\delta=(\delta_1, \dots, \delta_h) \in \{\pm 1\}^h$ is an $(h+1)$-tuple of integer partitions $(\lambda^0, \dots, \lambda^h)$ such that $\lambda^0=\lambda^h$ and $\lambda^{j-1} \succeq \lambda^j$ (resp. $\lambda^{j-1} \preceq \lambda^j$) if $\delta_j =-1$ (resp. if $\delta_j=1$).  We define the {\it size} of $\lambda$ as
$$
|\lambda|:=\sum_{j=0}^{h-1} |\lambda^j|.
$$
We define the {\it rank} of a cylindric partition $\lambda$ to be the number of $(-1)$'s in the profile.  We write $\mathcal{CP}_{\delta}$ for the set of cylindric partitions of profile $\delta$, including the ``empty cylindric partition'' of profile $\delta.$
\end{definition}

\begin{figure}[ht!]
\centering
    \begin{tikzpicture}[scale=.6, domain=0:2,every node/.style={minimum size=.60cm-\pgflinewidth, outer sep=0pt}]
\draw (3,0) grid (5,-1);
\draw (2,-1) grid (6,-2);
\draw (2,-2) grid (7,-3);
\draw (1,-3) grid (6,-4);
\draw (0,-4) grid (5,-5);
\draw (1,-5) grid (4,-6);
\draw (2,-6) grid (3,-7);
\node[rectangle,fill=yellow] at (4.5,-0.5){};
\node[rectangle,fill=yellow] at (5.5,-1.5){};
\node[rectangle,fill=yellow] at (0.5,-4.5){};
\node[rectangle,fill=yellow] at (1.5,-5.5){};
\node at (3.5,-.5) {{\Large 5}};
\node at (4.5,-.5) {{\Large 3}};
\node at (2.5,-1.5) {{\Large 4}};
\node at (3.5,-1.5) {{\Large 4}};
\node at (4.5,-1.5) {{\Large 2}};
\node at (5.5,-1.5) {{\Large 1}};
\node at (2.5,-2.5) {{\Large 4}};
\node at (3.5,-2.5) {{\Large 3}};
\node at (1.5,-3.5) {{\Large 5}};
\node at (2.5,-3.5) {{\Large 2}};
\node at (3.5,-3.5) {{\Large 2}};
\node at (0.5,-4.5) {{\Large 3}};
\node at (1.5,-4.5) {{\Large 2}};
\node at (2.5,-4.5) {{\Large 1}};
\node at (1.5,-5.5) {{\Large 1}};
\draw [-, line width=1mm, color=purple] (0,-4) -- (1,-4) -- (1,-3) -- (2,-3) -- (2,-1) -- (3,-1) -- (3, 0) -- (4,0);
\end{tikzpicture}

  \caption{A cylindric partition of width $9,$ profile
  $(-1,1,-1,1,1,-1,1,-1)$, rank 4 and size 42.  Here $\lambda^0$ and $\lambda^h$ are highlighted.  The purple line indicates the profile, from bottom-left to top-right.}
  
  \label{F:cylindric}
  
  \end{figure}

Graphically, a profile $\delta$ starts on the top left corner of the starting diagonal $\lambda^0$ and ends on the top left corner of the last diagonal $\lambda^h$.

 Borodin proved that the generating function for cylindric partitions of any profile is always an infinite product.  We use the formulation in \cite{HX}.

\begin{proposition}[\cite{B}, Proposition on p. 8]\label{P:borodin}
Let $\delta=(\delta_1, \dots, \delta_h)$ be a profile and define the multiset
$$
W_3(\delta):=\{h\} \cup \{j-i : 1 \leq i < j \leq h, \delta_i > \delta_j\} \cup \{h-(j-i):1 \leq i < j \leq h, \delta_i < \delta_j\}.
$$
Then
\begin{equation}\label{E:Borodin}
\mathrm{CP}_{\delta}(q):=\sum_{\lambda \in \mathcal{ CP}_{\delta}} q^{|\lambda|} = \prod_{k \in W_3(\delta)}\frac{1}{(q^k;q^h)_{\infty}}.
\end{equation}
\end{proposition}

\begin{remark}
This product is always symmetric---that is, every $k<h$ appears as many times in $W_3(\delta)$ as $h-k$.  So, up to a rational power of $q$, $\mathrm{CP}_{\delta}(q)$ is a modular form.  To the best of our knowledge a proof of this non-trivial fact was lacking in the literature, so we provide one in Appendix A.
\end{remark}

To look for ``sum-side'' companions to the above products, Corteel and Welsh introduced the two variable generating function, $
\mathrm{CP}_{\delta}(z;q):=\sum_{\lambda \in \mathcal{CP}_{\delta}} z^{\mathrm{max}(\lambda)}q^{|\lambda|},
$
where $\mathrm{max}(\lambda)$ denotes the size of the largest part among all partitions in $\lambda$.  They proved general systems of recurrences for these functions and found sum-generating functions that solved all cases for width 7 and rank 3.  The identity \eqref{E:CorteelWelshwidth7rank3} in the introduction is $(q;q)_{\infty}\mathrm{CP}_{\delta}(1;q)$ with $\delta=(-1,-1,-1,1,1,1,1)$.  Identities for all rank 3, width 8 cases were subsequently proved by Corteel, Dousse and the second author in \cite{CDU}.  Warnaar \cite{W} further conjectured identities for all rank 3, width $h$ profiles with $3 \nmid h$.

To allow for products much more general than those in \eqref{E:Borodin}, we consider one natural restriction and one generalization of cylindric partitions defined by Han and Xiong in \cite{HX}.
\begin{definition}
A {\it symmetric cylindric partition} $\lambda$ is a cylindric partition of the form \newline $(\lambda^h, \dots, \lambda^1, \lambda^0, \lambda^1, \dots, \lambda^h)$ with profile $\delta=(-\delta_h, \dots, -\delta_1, \delta_1, \dots, \delta_h).$  We define the {\it size} as
$$
|\lambda|:=|\lambda^0|+|\lambda^h|+2\sum_{j=1}^{h-1} |\lambda^j|.
$$
Let $\mathcal{SCP}_{\delta}$ be the set of symmetric cylindric partitions of profile $\delta$, including the ``empty partition''.
\end{definition}
\begin{remark}
Han and Xiong defined the size of symmetric cylindric partitions instead as $|\lambda|:=|\lambda^0|+2\sum_{j=1}^{h} |\lambda^j|$ and stated their product formula for this notion of size.  But we will consider general weighted sizes below that include both of these.
\end{remark}

\begin{figure}[ht!]
\centering
    \begin{tikzpicture}[scale=.6, domain=0:2,every node/.style={minimum size=.60cm-\pgflinewidth, outer sep=0pt}]
\draw (3,0) grid (4,-1);
\draw (1,-1) grid (5,-2);
\draw (1,-2) grid (6,-3);
\draw (0,-3) grid (5,-4);
\draw (1,-4) grid (4,-5);
\draw (2,-5) grid (3,-6);
\node[rectangle,fill=yellow] at (3.5,-0.5){};
\node[rectangle,fill=yellow] at (4.5,-1.5){};
\node[rectangle,fill=yellow] at (0.5,-3.5){};
\node[rectangle,fill=yellow] at (1.5,-4.5){};
\node[rectangle,fill=green] at (1.5,-1.5){};
\node[rectangle,fill=green] at (2.5,-2.5){};
\node[rectangle,fill=green] at (3.5,-3.5){};
\node at (3.5,-0.5) {{\Large 4}};
\node at (1.5,-1.5) {{\Large 5}};
\node at (2.5,-1.5) {{\Large 5}};
\node at (3.5,-1.5) {{\Large 3}};
\node at (4.5,-1.5) {{\Large 1}};
\node at (1.5,-2.5) {{\Large 5}};
\node at (2.5,-2.5) {{\Large 2}};
\node at (3.5,-2.5) {{\Large 2}};
\node at (0.5,-3.5) {{\Large 4}};
\node at (1.5,-3.5) {{\Large 3}};
\node at (2.5,-3.5) {{\Large 2}};
\node at (3.5,-3.5) {{\Large 1}};
\node at (1.5,-4.5) {{\Large 1}};
\draw [-, line width=1mm, color=purple] (0,-3) -- (1,-3) -- (1,-1) -- (3,-1) -- (3,0);
\end{tikzpicture}

 \caption{A symmetric cylindric partition of width $6,$ profile
  $(-1,1,1,-1,-1,1)$ and size 38.  The green highlighted partition indicates the line of symmetry.}
  \label{F:symmetriccylindric}
 \end{figure}
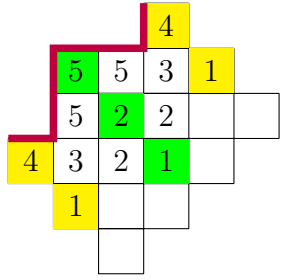

\begin{remark}
When viewing a symmetric cylindric partition on the face of a cylinder, one notices that there are two axes of symmetry, namely at $\lambda^0$ and at $\lambda^h$.  Thus, it is easy to see that $\mathrm{SCP}_{\delta}(q)=\mathrm{SCP}_{-\mathrm{rev}(\delta)}(q),$ where $\mathrm{rev}(\delta)$ is the reverse of $\delta.$
\end{remark}

If the restriction $\lambda^0=\lambda^h$ which creates a repeating pattern in a cylindric partition is dropped, then the resulting structures were called skew double shifted plane partitions by Han and Xiong \cite{HX}.
\begin{definition}
A {\it skew double shifted plane partition} (DSPP) of {\it width} $h$ with {\it profile} $\delta=(\delta_1, \dots, \delta_h) \in \{\pm 1\}^h$ is an $(h+1)$-tuple of integer partitions $(\lambda^0, \dots, \lambda^h)$ such that $\lambda^{j-1} \succeq \lambda^j$ (resp. $\lambda^{j-1} \preceq \lambda^j$) if $\delta_j =-1$ (resp. if $\delta_j=1$).  We define the {\it size} as
$$
|\lambda|:=\sum_{j=0}^{h} |\lambda^j|.
$$
Let $\mathcal{DSPP}_{\delta}$ be the set of skew double shifted plane partitions of profile $\delta$, including the ``empty partition''.
\end{definition}

\begin{figure}[h!]
\centering
    \begin{tikzpicture}[scale=.6, domain=0:2,every node/.style={minimum size=.60cm-\pgflinewidth, outer sep=0pt}]
\draw (4,0) grid (5,-1);
\draw (2,-1) grid (6,-2);
\draw (0,-2) grid (7,-3);
\draw (1,-3) grid (7,-4);
\draw (2,-4) grid (6,-5);
\draw (3,-5) grid (5,-6);
\node at (4.5,-0.5) {{\Large 7}};
\node at (2.5,-1.5) {{\Large 7}};
\node at (3.5,-1.5) {{\Large 6}};
\node at (4.5,-1.5) {{\Large 4}};
\node at (5.5,-1.5) {{\Large 4}};
\node at (0.5,-2.5) {{\Large 6}};
\node at (1.5,-2.5) {{\Large 5}};
\node at (2.5,-2.5) {{\Large 5}};
\node at (3.5,-2.5) {{\Large 3}};
\node at (4.5,-2.5) {{\Large 2}};
\node at (5.5,-2.5) {{\Large 1}};
\node at (1.5,-3.5) {{\Large 4}};
\node at (2.5,-3.5) {{\Large 3}};
\node at (3.5,-3.5) {{\Large 1}};
\node at (2.5,-4.5) {{\Large 2}};
\node at (3.5,-4.5) {{\Large 1}};
\draw [-, line width=1mm, color=purple] (0,-2) -- (2,-2) -- (2,-1) -- (4,-1) -- (4,0);
\end{tikzpicture}

  \caption{A DSPP of width $6,$ profile
  $(-1,-1,1,-1,-1,1)$ and size 61.}
  \label{F:dspp}
  
  \end{figure}
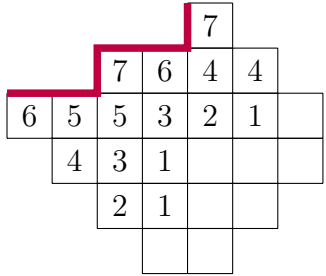 

Han and Xiong \cite{HX} proved a general lemma in the theory of symmetric functions and used it to prove product formulas analogous to Borodin's for both symmetric cylindric partitions  and DSPPs.  It follows directly from their work that generating functions are products even with a general weighted size.  For a DSPP $\lambda$ of width $h$ and a vector ${\bf a}=(a_0, \dots, a_h) \in \mathbb{R}_{\geq 0}^{h+1}$, we define 
$$
|\lambda|_{{\bf a}}:=\sum_{j=0}^h a_j|\lambda^j|.
$$
Similarly, for a cylindric partition $\lambda$ of width $h$ and a vector ${\bf a}=(a_0, \dots, a_{h-1}) \in \mathbb{R}_{\geq 0}^h$, we define
$$
|\lambda|_{{\bf a}}:=\sum_{j=0}^{h-1}a_j|\lambda^{j}|.
$$
Throughout, let $A_k:=\sum_{j=0}^{k-1} a_j.$  We say that ${\bf a}$ is the {\it standard weight} when ${\bf a}=(1,1, \dots,1)$.
\begin{proposition}\label{P:CPgeneralproduct}
Let $\delta=(\delta_1, \dots, \delta_h)$ be a profile of width $h$, and let ${\bf a}=(a_0, \dots, a_{h-1}) \in \mathbb{R}_{\geq 0}^h$.  Define the set
\begin{align*}
    W^{{\bf a}}_3(\delta)&:=\{A_h\} \cup \{A_j-A_i : 1 \leq i < j \leq h, \delta_i < \delta_j\} \cup \{A_h-(A_j-A_i):1 \leq i < j \leq h, \delta_i > \delta_j\}.
\end{align*}
If $0 \notin W_3^{{\bf a}}(\delta)$, then
\begin{equation}\label{E:Borodingeneral}
\mathrm{CP}_{\delta}^{{\bf a}}(q):=\sum_{\lambda \in \mathcal{ CP}_{\delta}} q^{|\lambda|_{{\bf a}}} = \prod_{k \in W^{{\bf a}}_3(\delta)}\frac{1}{(q^{k};q^{A_h})_{\infty}}.
\end{equation}
\end{proposition}

The proof of Proposition \ref{P:CPgeneralproduct} directly follows from  \cite[Theorem 3.1]{HX} by setting $u_i:=q^{a_i}.$ In the same spirit, we get the following.

\begin{proposition}\label{P:DSPPproductgeneral}
Let $\delta=(\delta_1, \dots, \delta_h)$ be a profile of width $h$, and let ${\bf a}=(a_0, \dots, a_{h}) \in \mathbb{R}_{\geq 0}^{h+1}$.  Define the sets
\begin{align*}
W^{{\bf a}}_1(\delta):=&\{A_{h+1}\} \cup \{A_i : \delta_i=-1\} \cup \{A_{h+1}-A_i : \delta_i=1\}, \\
W^{{\bf a}}_2(\delta):=&\{A_i+A_j : 1 \leq i < j \leq h, \ \delta_i=\delta_j=-1\} \\
&\cup \{2A_{h+1}-A_i-A_j : 1 \leq i < j \leq h, \ \delta_i=\delta_j=1\} \\
& \cup \{2A_{h+1}-(A_j-A_i) : 1 \leq i < j \leq h, \ \delta_i<\delta_j\} \\
& \cup \{A_j-A_i : 1 \leq i < j \leq h, \ \delta_i>\delta_j\}.
\end{align*}
If $0 \notin W_1^{{\bf a}}(\delta) \cup W_2^{{\bf a}}(\delta)$, then
\begin{equation}\label{E:DSPPproductgeneral}
\mathrm{DSPP}_{\delta}^{{\bf a}}(q):=\sum_{\lambda \in \mathcal{DSPP}_{\delta}} q^{|\lambda|_{{\bf a}}} = \prod_{\substack{k \in W^{{\bf a}}_1(\delta) \\ \ell \in W^{{\bf a}}_2(\delta)}} \frac{1}{(q^{k};q^{A_{h+1}})_{\infty}(q^{\ell};q^{2A_{h+1}})_{\infty}}.
\end{equation}
\end{proposition}

Note that for the width 0 and 1 profiles these generating functions with standard weights equals the partition function \[\mathrm{DSPP}_{\emptyset}^{(1)}(q)=\mathrm{DSPP}_{(1)}^{(1,1)}(q)=\frac{1}{(q;q)_{\infty}}.\] This corresponence can easily be seen through the diagrams of the counted objects. 

\begin{remark}
In contrast to the generating functions for cylindric partitions, neither of the products \eqref{E:Borodingeneral} or \eqref{E:DSPPproductgeneral} are symmetric.
\end{remark}

\begin{remark}\label{R:differentproducts}
We may now write {\it any} product of the following form as a weighted cylindric partition generating function; precisely, for any real numbers  $0 < b_1 \leq  b_2 \leq \dots \leq b_{r+1}$, we have
$$
\mathrm{CP}_{(-1,-1,\dots, -1,1)}^{(b_1, b_2-b_1, \dots, b_{r+1}-b_r)}(q)=\frac{1}{(q^{b_1},q^{b_2},\dots ,q^{b_r},q^{b_{r+1}};q^{b_{r+1}})_{\infty}}.
$$
For example, this approach provides a combinatorial interpretation for a specilization of the reciprocal of the Ramanujan theta function  \cite[(1.59). p.37]{Sills} \[\sum_{n=-\infty}^\infty a^{n+1\choose 2}b^{n\choose 2} = (-a,-b,ab;ab)_\infty,\] as a generating function for the number of weighted cylindric partitions with profile $(1,1,-1)$. Precisely, let $b_2>b_1>0$ then \[\mathrm{CP}_{(-1,-1,1)}^{( b_1 b_2-b_1,b_1)}(q)=\frac{1}{(q^{b_1},q^{b_2},q^{b_1+b_2};q^{b_1+b_2})_{\infty}} = \left(\sum_{n=-\infty}^\infty (-1)^n q^{b_1{n+1\choose 2} + b_2{n\choose 2}} \right)^{-1}.\]

The profile $(-1,-1, \dots,-1,-1, 1)$ of width $h$ has rank $h-1$ and does not lead to interesting recurrences and sum-sides in general; however, there are often multiple combinations of weights and profiles with higher rank that yield the same product. In general, by picking different profiles and weights we can discover connections between weighted counts of different classes of cylindric partitions. To demonstrate, from Proposition~\ref{P:CPgeneralproduct}, we can easily confirm that
\begin{align}\label{E:CPs1}\mathrm{CP}_{(-1,-1,1)}^{(1,3,1)}(q) &= \mathrm{CP}_{(-1,-1,1,1)}^{(1,3,1,5)}(q)=\mathrm{CP}_{(-1,1,-1,1,1)}^{(1,4,4,1,5)}(q)=\mathrm{CP}_{(-1,1,-1,1,-1)}^{(5,4,1,1,4)}(q)=  \frac{(q^2,q^3;q^5)}{(q;q)_\infty},\\
\label{E:CPs2}\mathrm{CP}_{(-1,1,1)}^{(2,2,1)}(q) &= \mathrm{CP}_{(-1,1,-1,1)}^{(2,2,3,3)}(q)=\mathrm{CP}_{(-1,1,-1,1,1)}^{(2,3,3,2,5)}(q)=\mathrm{CP}_{(-1,1,-1,1,-1)}^{(5,3,2,2,3)}(q)=  \frac{(q,q^4;q^5)}{(q;q)_\infty}.
\end{align}

Similar to finding weighted correspondences between cylindric partitions with different profiles, we can also find weighted correspondences between skew double shifted plane partitions with different profiles. Furthermore, we can find weighted relations between CPs and DSPPs. 

For example, Propositions~\ref{P:CPgeneralproduct} and \ref{P:DSPPproductgeneral} is enough to see that \begin{equation}
    \label{E:CP_DSPP_weighted_example} \mathrm{CP}^{(1,1,0,0)}_{(-1,1,1,1)}(q) = \mathrm{DSPP}^{(0,1,0)}_{(1,-1)}(q) = \frac{1}{(q;q)^3_\infty (q;q^2)_\infty}.
\end{equation}

To give an explicit example, we note that the coefficient of the $q^3$ term of the $q$-series of \eqref{E:CP_DSPP_weighted_example} is 36. We explicitly present these weighted CPs and DSPPs in Figure~\ref{F:CP_DSPP_figure}, where the yellow diagonals for the cylindric partitions highlight the same diagonal's repetition due to cylindricity (only one of the diagonals contribute to the total size) and the gray diagonals are weighted with 0 (the numbers on these diagonals do not contribute to the total size of the object). The total value in the boxes with the white  backgrounds (each counted with weight 1 in this case) makes the size of these objects. There are some explicitly stated 0 parts. On top of that every empty box in these diagrams can be thought to have 0s inside and they add nothing to the total size. Finally, we compress multiple objects with braces of a range of possible numbers to avoid repetition the presentation.

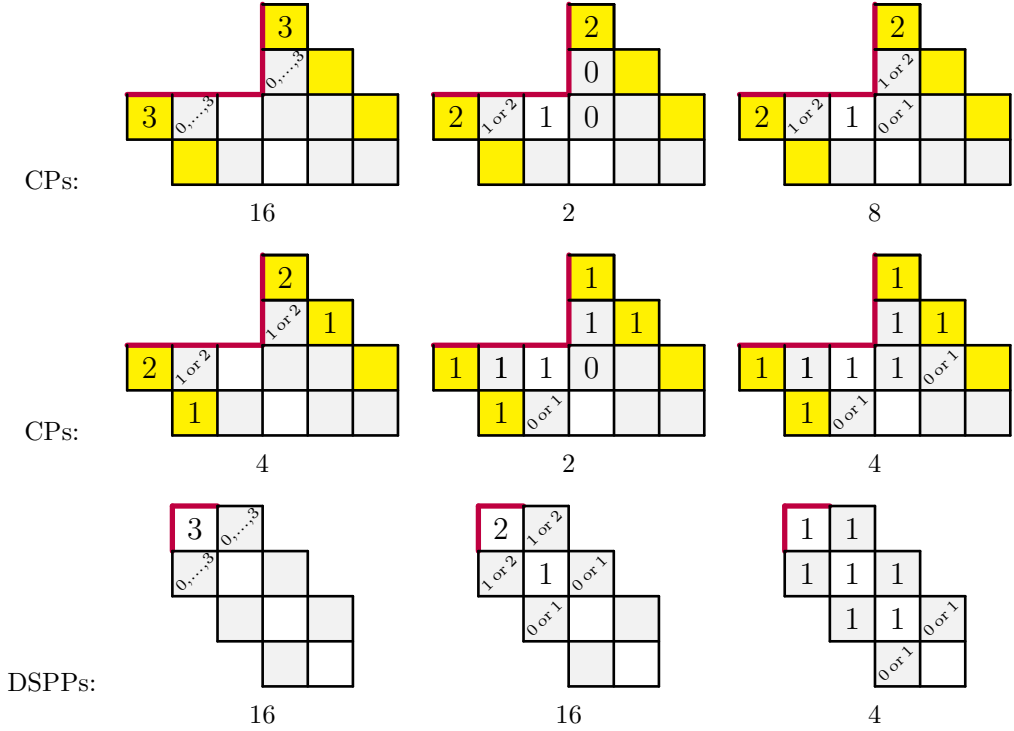
\begin{figure}[ht!]
\begin{tabular}{cccc}
CPs:&\begin{tikzpicture}[scale=0.6,line cap=round,line join=round,x=1cm,y=1cm]
\clip(1.9,1.9) rectangle (8.1,6.1);
\fill[line width=0pt,color=yellow,fill=yellow,fill opacity=1] (5,6) -- (6,6) -- (6,5) -- (5,5) -- cycle;
\fill[line width=2pt,color=yellow,fill=yellow,fill opacity=1] (6,5) -- (7,5) -- (7,4) -- (6,4) -- cycle;
\fill[line width=2pt,color=yellow,fill=yellow,fill opacity=1] (7,4) -- (8,4) -- (8,3) -- (7,3) -- cycle;
\fill[line width=2pt,color=yellow,fill=yellow,fill opacity=1] (2,4) -- (2,3) -- (3,3) -- (3,4) -- cycle;
\fill[line width=2pt,color=yellow,fill=yellow,fill opacity=1] (3,3) -- (4,3) -- (4,2) -- (3,2) -- cycle;
\fill[line width=2pt,color=gray,fill=gray,fill opacity=0.1] (5,5) -- (5,3) -- (6,3) -- (6,2) -- (8,2) -- (8,3) -- (7,3) -- (7,4) -- (6,4) -- (6,5) -- cycle;
\fill[line width=2pt,color=gray,fill=gray,fill opacity=0.1] (3,4) -- (4,4) -- (4,3) -- (3,3) -- cycle;
\fill[line width=2pt,color=gray,fill=gray,fill opacity=0.1] (4,3) -- (5,3) -- (5,2) -- (4,2) -- cycle;
\draw [line width=2pt,color=purple] (2,4)-- (5,4);
\draw [line width=2pt,color=purple] (5,4)-- (5,6);
\draw [line width=1pt] (2,3)-- (2,4);
\draw [line width=1pt] (5,6)-- (6,6);
\draw [line width=1pt] (6,6)-- (6,2);
\draw [line width=1pt] (5,5)-- (7,5);
\draw [line width=1pt] (5,4)-- (8,4);
\draw [line width=1pt] (7,5)-- (7,2);
\draw [line width=1pt] (2,3)-- (8,3);
\draw [line width=1pt] (8,4)-- (8,2);
\draw [line width=1pt] (5,4)-- (5,2);
\draw [line width=1pt] (4,4)-- (4,2);
\draw [line width=1pt] (3,4)-- (3,2);
\draw [line width=1pt] (5,6)-- (6,6);
\draw [line width=1pt] (6,6)-- (6,5);
\draw [line width=1pt] (6,5)-- (5,5);
\draw [line width=1pt] (3,2)-- (8,2);
\node at (2.5,3.5) {{\Large 3}};
\node at (5.5,5.5) {{\Large 3}};
\node at (3.2,3.1) {{\begin{rotate}{45}\tiny 0,...,3 \end{rotate}}};
\node at (5.2,4.1) {{\begin{rotate}{45}\tiny 0,...,3 \end{rotate}}};
\end{tikzpicture}
&
\begin{tikzpicture}[scale=0.6,line cap=round,line join=round,x=1cm,y=1cm]
\clip(1.9,1.9) rectangle (8.1,6.1);
\fill[line width=0pt,color=yellow,fill=yellow,fill opacity=1] (5,6) -- (6,6) -- (6,5) -- (5,5) -- cycle;
\fill[line width=2pt,color=yellow,fill=yellow,fill opacity=1] (6,5) -- (7,5) -- (7,4) -- (6,4) -- cycle;
\fill[line width=2pt,color=yellow,fill=yellow,fill opacity=1] (7,4) -- (8,4) -- (8,3) -- (7,3) -- cycle;
\fill[line width=2pt,color=yellow,fill=yellow,fill opacity=1] (2,4) -- (2,3) -- (3,3) -- (3,4) -- cycle;
\fill[line width=2pt,color=yellow,fill=yellow,fill opacity=1] (3,3) -- (4,3) -- (4,2) -- (3,2) -- cycle;
\fill[line width=2pt,color=gray,fill=gray,fill opacity=0.1] (5,5) -- (5,3) -- (6,3) -- (6,2) -- (8,2) -- (8,3) -- (7,3) -- (7,4) -- (6,4) -- (6,5) -- cycle;
\fill[line width=2pt,color=gray,fill=gray,fill opacity=0.1] (3,4) -- (4,4) -- (4,3) -- (3,3) -- cycle;
\fill[line width=2pt,color=gray,fill=gray,fill opacity=0.1] (4,3) -- (5,3) -- (5,2) -- (4,2) -- cycle;
\draw [line width=2pt,color=purple] (2,4)-- (5,4);
\draw [line width=2pt,color=purple] (5,4)-- (5,6);
\draw [line width=1pt] (2,3)-- (2,4);
\draw [line width=1pt] (5,6)-- (6,6);
\draw [line width=1pt] (6,6)-- (6,2);
\draw [line width=1pt] (5,5)-- (7,5);
\draw [line width=1pt] (5,4)-- (8,4);
\draw [line width=1pt] (7,5)-- (7,2);
\draw [line width=1pt] (2,3)-- (8,3);
\draw [line width=1pt] (8,4)-- (8,2);
\draw [line width=1pt] (5,4)-- (5,2);
\draw [line width=1pt] (4,4)-- (4,2);
\draw [line width=1pt] (3,4)-- (3,2);
\draw [line width=1pt] (5,6)-- (6,6);
\draw [line width=1pt] (6,6)-- (6,5);
\draw [line width=1pt] (6,5)-- (5,5);
\draw [line width=1pt] (3,2)-- (8,2);
\node at (2.5,3.5) {{\Large 2}};
\node at (5.5,5.5) {{\Large 2}};
\node at (3.2,3.1) {{\begin{rotate}{45}\tiny 1\,\text{or}\,2 \end{rotate}}};
\node at (4.5,3.5) {{\Large 1}};
\node at (5.5,4.5) {{\Large 0}};
\node at (5.5,3.5) {{\Large 0}};
\end{tikzpicture}
&
\begin{tikzpicture}[scale=0.6,line cap=round,line join=round,x=1cm,y=1cm]
\clip(1.9,1.9) rectangle (8.1,6.1);
\fill[line width=0pt,color=yellow,fill=yellow,fill opacity=1] (5,6) -- (6,6) -- (6,5) -- (5,5) -- cycle;
\fill[line width=2pt,color=yellow,fill=yellow,fill opacity=1] (6,5) -- (7,5) -- (7,4) -- (6,4) -- cycle;
\fill[line width=2pt,color=yellow,fill=yellow,fill opacity=1] (7,4) -- (8,4) -- (8,3) -- (7,3) -- cycle;
\fill[line width=2pt,color=yellow,fill=yellow,fill opacity=1] (2,4) -- (2,3) -- (3,3) -- (3,4) -- cycle;
\fill[line width=2pt,color=yellow,fill=yellow,fill opacity=1] (3,3) -- (4,3) -- (4,2) -- (3,2) -- cycle;
\fill[line width=2pt,color=gray,fill=gray,fill opacity=0.1] (5,5) -- (5,3) -- (6,3) -- (6,2) -- (8,2) -- (8,3) -- (7,3) -- (7,4) -- (6,4) -- (6,5) -- cycle;
\fill[line width=2pt,color=gray,fill=gray,fill opacity=0.1] (3,4) -- (4,4) -- (4,3) -- (3,3) -- cycle;
\fill[line width=2pt,color=gray,fill=gray,fill opacity=0.1] (4,3) -- (5,3) -- (5,2) -- (4,2) -- cycle;
\draw [line width=2pt,color=purple] (2,4)-- (5,4);
\draw [line width=2pt,color=purple] (5,4)-- (5,6);
\draw [line width=1pt] (2,3)-- (2,4);
\draw [line width=1pt] (5,6)-- (6,6);
\draw [line width=1pt] (6,6)-- (6,2);
\draw [line width=1pt] (5,5)-- (7,5);
\draw [line width=1pt] (5,4)-- (8,4);
\draw [line width=1pt] (7,5)-- (7,2);
\draw [line width=1pt] (2,3)-- (8,3);
\draw [line width=1pt] (8,4)-- (8,2);
\draw [line width=1pt] (5,4)-- (5,2);
\draw [line width=1pt] (4,4)-- (4,2);
\draw [line width=1pt] (3,4)-- (3,2);
\draw [line width=1pt] (5,6)-- (6,6);
\draw [line width=1pt] (6,6)-- (6,5);
\draw [line width=1pt] (6,5)-- (5,5);
\draw [line width=1pt] (3,2)-- (8,2);
\node at (2.5,3.5) {{\Large 2}};
\node at (5.5,5.5) {{\Large 2}};
\node at (3.2,3.1) {{\begin{rotate}{45}\tiny 1\,\text{or}\,2 \end{rotate}}};
\node at (4.5,3.5) {{\Large 1}};
\node at (5.2,4.1) {{\begin{rotate}{45}\tiny 1\,\text{or}\,2 \end{rotate}}};
\node at (5.2,3.1) {{\begin{rotate}{45}\tiny 0\,\text{or}\,1 \end{rotate}}};
\end{tikzpicture}
\\
&16 & 2 & 8\\ [-1ex]\\
CPs: &\begin{tikzpicture}[scale=0.6,line cap=round,line join=round,x=1cm,y=1cm]
\clip(1.9,1.9) rectangle (8.1,6.1);
\fill[line width=0pt,color=yellow,fill=yellow,fill opacity=1] (5,6) -- (6,6) -- (6,5) -- (5,5) -- cycle;
\fill[line width=2pt,color=yellow,fill=yellow,fill opacity=1] (6,5) -- (7,5) -- (7,4) -- (6,4) -- cycle;
\fill[line width=2pt,color=yellow,fill=yellow,fill opacity=1] (7,4) -- (8,4) -- (8,3) -- (7,3) -- cycle;
\fill[line width=2pt,color=yellow,fill=yellow,fill opacity=1] (2,4) -- (2,3) -- (3,3) -- (3,4) -- cycle;
\fill[line width=2pt,color=yellow,fill=yellow,fill opacity=1] (3,3) -- (4,3) -- (4,2) -- (3,2) -- cycle;
\fill[line width=2pt,color=gray,fill=gray,fill opacity=0.1] (5,5) -- (5,3) -- (6,3) -- (6,2) -- (8,2) -- (8,3) -- (7,3) -- (7,4) -- (6,4) -- (6,5) -- cycle;
\fill[line width=2pt,color=gray,fill=gray,fill opacity=0.1] (3,4) -- (4,4) -- (4,3) -- (3,3) -- cycle;
\fill[line width=2pt,color=gray,fill=gray,fill opacity=0.1] (4,3) -- (5,3) -- (5,2) -- (4,2) -- cycle;
\draw [line width=2pt,color=purple] (2,4)-- (5,4);
\draw [line width=2pt,color=purple] (5,4)-- (5,6);
\draw [line width=1pt] (2,3)-- (2,4);
\draw [line width=1pt] (5,6)-- (6,6);
\draw [line width=1pt] (6,6)-- (6,2);
\draw [line width=1pt] (5,5)-- (7,5);
\draw [line width=1pt] (5,4)-- (8,4);
\draw [line width=1pt] (7,5)-- (7,2);
\draw [line width=1pt] (2,3)-- (8,3);
\draw [line width=1pt] (8,4)-- (8,2);
\draw [line width=1pt] (5,4)-- (5,2);
\draw [line width=1pt] (4,4)-- (4,2);
\draw [line width=1pt] (3,4)-- (3,2);
\draw [line width=1pt] (5,6)-- (6,6);
\draw [line width=1pt] (6,6)-- (6,5);
\draw [line width=1pt] (6,5)-- (5,5);
\draw [line width=1pt] (3,2)-- (8,2);
\node at (2.5,3.5) {{\Large 2}};
\node at (3.5,2.5) {{\Large 1}};
\node at (5.5,5.5) {{\Large 2}};
\node at (6.5,4.5) {{\Large 1}};
\node at (3.2,3.1) {{\begin{rotate}{45}\tiny 1\,\text{or}\,2 \end{rotate}}};
\node at (5.2,4.1) {{\begin{rotate}{45}\tiny 1\,\text{or}\,2 \end{rotate}}};
\end{tikzpicture}
&
\begin{tikzpicture}[scale=0.6,line cap=round,line join=round,x=1cm,y=1cm]
\clip(1.9,1.9) rectangle (8.1,6.1);
\fill[line width=0pt,color=yellow,fill=yellow,fill opacity=1] (5,6) -- (6,6) -- (6,5) -- (5,5) -- cycle;
\fill[line width=2pt,color=yellow,fill=yellow,fill opacity=1] (6,5) -- (7,5) -- (7,4) -- (6,4) -- cycle;
\fill[line width=2pt,color=yellow,fill=yellow,fill opacity=1] (7,4) -- (8,4) -- (8,3) -- (7,3) -- cycle;
\fill[line width=2pt,color=yellow,fill=yellow,fill opacity=1] (2,4) -- (2,3) -- (3,3) -- (3,4) -- cycle;
\fill[line width=2pt,color=yellow,fill=yellow,fill opacity=1] (3,3) -- (4,3) -- (4,2) -- (3,2) -- cycle;
\fill[line width=2pt,color=gray,fill=gray,fill opacity=0.1] (5,5) -- (5,3) -- (6,3) -- (6,2) -- (8,2) -- (8,3) -- (7,3) -- (7,4) -- (6,4) -- (6,5) -- cycle;
\fill[line width=2pt,color=gray,fill=gray,fill opacity=0.1] (3,4) -- (4,4) -- (4,3) -- (3,3) -- cycle;
\fill[line width=2pt,color=gray,fill=gray,fill opacity=0.1] (4,3) -- (5,3) -- (5,2) -- (4,2) -- cycle;
\draw [line width=2pt,color=purple] (2,4)-- (5,4);
\draw [line width=2pt,color=purple] (5,4)-- (5,6);
\draw [line width=1pt] (2,3)-- (2,4);
\draw [line width=1pt] (5,6)-- (6,6);
\draw [line width=1pt] (6,6)-- (6,2);
\draw [line width=1pt] (5,5)-- (7,5);
\draw [line width=1pt] (5,4)-- (8,4);
\draw [line width=1pt] (7,5)-- (7,2);
\draw [line width=1pt] (2,3)-- (8,3);
\draw [line width=1pt] (8,4)-- (8,2);
\draw [line width=1pt] (5,4)-- (5,2);
\draw [line width=1pt] (4,4)-- (4,2);
\draw [line width=1pt] (3,4)-- (3,2);
\draw [line width=1pt] (5,6)-- (6,6);
\draw [line width=1pt] (6,6)-- (6,5);
\draw [line width=1pt] (6,5)-- (5,5);
\draw [line width=1pt] (3,2)-- (8,2);
\node at (2.5,3.5) {{\Large 1}};
\node at (3.5,2.5) {{\Large 1}};
\node at (3.5,3.5) {{\Large 1}};
\node at (4.5,3.5) {{\Large 1}};
\node at (5.5,3.5) {{\Large 0}};
\node at (4.2,2.1)  {{\begin{rotate}{45}\tiny 0\,\text{or}\,1 \end{rotate}}};
\node at (5.5,5.5) {{\Large 1}};
\node at (6.5,4.5) {{\Large 1}};
\node at (3.5,3.5) {{\Large 1}};
\node at (5.5,4.5) {{\Large 1}};
\end{tikzpicture}
&
\begin{tikzpicture}[scale=0.6,line cap=round,line join=round,x=1cm,y=1cm]
\clip(1.9,1.9) rectangle (8.1,6.1);
\fill[line width=0pt,color=yellow,fill=yellow,fill opacity=1] (5,6) -- (6,6) -- (6,5) -- (5,5) -- cycle;
\fill[line width=2pt,color=yellow,fill=yellow,fill opacity=1] (6,5) -- (7,5) -- (7,4) -- (6,4) -- cycle;
\fill[line width=2pt,color=yellow,fill=yellow,fill opacity=1] (7,4) -- (8,4) -- (8,3) -- (7,3) -- cycle;
\fill[line width=2pt,color=yellow,fill=yellow,fill opacity=1] (2,4) -- (2,3) -- (3,3) -- (3,4) -- cycle;
\fill[line width=2pt,color=yellow,fill=yellow,fill opacity=1] (3,3) -- (4,3) -- (4,2) -- (3,2) -- cycle;
\fill[line width=2pt,color=gray,fill=gray,fill opacity=0.1] (5,5) -- (5,3) -- (6,3) -- (6,2) -- (8,2) -- (8,3) -- (7,3) -- (7,4) -- (6,4) -- (6,5) -- cycle;
\fill[line width=2pt,color=gray,fill=gray,fill opacity=0.1] (3,4) -- (4,4) -- (4,3) -- (3,3) -- cycle;
\fill[line width=2pt,color=gray,fill=gray,fill opacity=0.1] (4,3) -- (5,3) -- (5,2) -- (4,2) -- cycle;
\draw [line width=2pt,color=purple] (2,4)-- (5,4);
\draw [line width=2pt,color=purple] (5,4)-- (5,6);
\draw [line width=1pt] (2,3)-- (2,4);
\draw [line width=1pt] (5,6)-- (6,6);
\draw [line width=1pt] (6,6)-- (6,2);
\draw [line width=1pt] (5,5)-- (7,5);
\draw [line width=1pt] (5,4)-- (8,4);
\draw [line width=1pt] (7,5)-- (7,2);
\draw [line width=1pt] (2,3)-- (8,3);
\draw [line width=1pt] (8,4)-- (8,2);
\draw [line width=1pt] (5,4)-- (5,2);
\draw [line width=1pt] (4,4)-- (4,2);
\draw [line width=1pt] (3,4)-- (3,2);
\draw [line width=1pt] (5,6)-- (6,6);
\draw [line width=1pt] (6,6)-- (6,5);
\draw [line width=1pt] (6,5)-- (5,5);
\draw [line width=1pt] (3,2)-- (8,2);
\node at (2.5,3.5) {{\Large 1}};
\node at (3.5,2.5) {{\Large 1}};
\node at (3.5,3.5) {{\Large 1}};
\node at (4.5,3.5) {{\Large 1}};
\node at (5.5,3.5) {{\Large 1}};
\node at (6.2,3.1) {{\begin{rotate}{45}\tiny 0\,\text{or}\,1 \end{rotate}}};
\node at (4.2,2.1) {{\begin{rotate}{45}\tiny 0\,\text{or}\,1 \end{rotate}}};
\node at (5.5,5.5) {{\Large 1}};
\node at (6.5,4.5) {{\Large 1}};
\node at (3.5,3.5) {{\Large 1}};
\node at (5.5,4.5) {{\Large 1}};
\end{tikzpicture}\\
&4 & 2 & 4\\[-1ex]\\
DSPPs: &\begin{tikzpicture}[scale=0.6, line cap=round,line join=round,x=1cm,y=1cm]
\clip(1.9,0.9) rectangle (6.1,5.1);
\fill[line width=2pt,color=gray,fill=gray,fill opacity=0.1] (2,4) -- (3,4) -- (3,3) -- (2,3) -- cycle;
\fill[line width=2pt,color=gray,fill=gray,fill opacity=0.1] (3,3) -- (4,3) -- (4,2) -- (3,2) -- cycle;
\fill[line width=2pt,color=gray,fill=gray,fill opacity=0.1] (4,2) -- (5,2) -- (5,1) -- (4,1) -- cycle;
\fill[line width=2pt,color=gray,fill=gray,fill opacity=0.1] (3,5) -- (4,5) -- (4,4) -- (3,4) -- cycle;
\fill[line width=2pt,color=gray,fill=gray,fill opacity=0.1] (4,4) -- (5,4) -- (5,3) -- (4,3) -- cycle;
\fill[line width=2pt,color=gray,fill=gray,fill opacity=0.1] (6,2) -- (6,3) -- (5,3) -- (5,2) -- cycle;
\draw [line width=2pt,color=purple] (2,4)-- (2,5);
\draw [line width=2pt,color=purple] (2,5)-- (3,5);
\draw [line width=1pt] (3,5)-- (3,2);
\draw [line width=1pt] (3,2)-- (6,2);
\draw [line width=1pt] (5,1)-- (5,4);
\draw [line width=1pt] (5,4)-- (2,4);
\draw [line width=1pt] (4,5)-- (3,5);
\draw [line width=1pt] (4,5)-- (4,1);
\draw [line width=1pt] (6,3)-- (2,3);
\draw [line width=1pt] (2,3)-- (2,4);
\draw [line width=1pt] (4,1)-- (6,1);
\draw [line width=1pt] (6,1)-- (6,3);
\node at (2.5,4.5) {{\Large 3}};
\node at (2.2,3.1) {{\begin{rotate}{45}\tiny 0,...,3 \end{rotate}}};
\node at (3.2,4.1) {{\begin{rotate}{45}\tiny 0,...,3 \end{rotate}}};
\end{tikzpicture}
&
\begin{tikzpicture}[scale=0.6, line cap=round,line join=round,x=1cm,y=1cm]
\clip(1.9,0.9) rectangle (6.1,5.1);
\fill[line width=2pt,color=gray,fill=gray,fill opacity=0.1] (2,4) -- (3,4) -- (3,3) -- (2,3) -- cycle;
\fill[line width=2pt,color=gray,fill=gray,fill opacity=0.1] (3,3) -- (4,3) -- (4,2) -- (3,2) -- cycle;
\fill[line width=2pt,color=gray,fill=gray,fill opacity=0.1] (4,2) -- (5,2) -- (5,1) -- (4,1) -- cycle;
\fill[line width=2pt,color=gray,fill=gray,fill opacity=0.1] (3,5) -- (4,5) -- (4,4) -- (3,4) -- cycle;
\fill[line width=2pt,color=gray,fill=gray,fill opacity=0.1] (4,4) -- (5,4) -- (5,3) -- (4,3) -- cycle;
\fill[line width=2pt,color=gray,fill=gray,fill opacity=0.1] (6,2) -- (6,3) -- (5,3) -- (5,2) -- cycle;
\draw [line width=2pt,color=purple] (2,4)-- (2,5);
\draw [line width=2pt,color=purple] (2,5)-- (3,5);
\draw [line width=1pt] (3,5)-- (3,2);
\draw [line width=1pt] (3,2)-- (6,2);
\draw [line width=1pt] (5,1)-- (5,4);
\draw [line width=1pt] (5,4)-- (2,4);
\draw [line width=1pt] (4,5)-- (3,5);
\draw [line width=1pt] (4,5)-- (4,1);
\draw [line width=1pt] (6,3)-- (2,3);
\draw [line width=1pt] (2,3)-- (2,4);
\draw [line width=1pt] (4,1)-- (6,1);
\draw [line width=1pt] (6,1)-- (6,3);
\node at (2.5,4.5) {{\Large 2}};
\node at (2.2,3.1) {{\begin{rotate}{45}\tiny 1\,\text{or}\,2 \end{rotate}}};
\node at (3.2,4.1) {{\begin{rotate}{45}\tiny 1\,\text{or}\,2 \end{rotate}}};
\node at (3.5,3.5) {{\Large 1}};
\node at (3.2,2.1) {{\begin{rotate}{45}\tiny 0\,\text{or}\,1 \end{rotate}}};
\node at (4.2,3.1) {{\begin{rotate}{45}\tiny 0\,\text{or}\,1 \end{rotate}}};
\end{tikzpicture}
&
\begin{tikzpicture}[scale=0.6, line cap=round,line join=round,x=1cm,y=1cm]
\clip(1.9,0.9) rectangle (6.1,5.1);
\fill[line width=2pt,color=gray,fill=gray,fill opacity=0.1] (2,4) -- (3,4) -- (3,3) -- (2,3) -- cycle;
\fill[line width=2pt,color=gray,fill=gray,fill opacity=0.1] (3,3) -- (4,3) -- (4,2) -- (3,2) -- cycle;
\fill[line width=2pt,color=gray,fill=gray,fill opacity=0.1] (4,2) -- (5,2) -- (5,1) -- (4,1) -- cycle;
\fill[line width=2pt,color=gray,fill=gray,fill opacity=0.1] (3,5) -- (4,5) -- (4,4) -- (3,4) -- cycle;
\fill[line width=2pt,color=gray,fill=gray,fill opacity=0.1] (4,4) -- (5,4) -- (5,3) -- (4,3) -- cycle;
\fill[line width=2pt,color=gray,fill=gray,fill opacity=0.1] (6,2) -- (6,3) -- (5,3) -- (5,2) -- cycle;
\draw [line width=2pt,color=purple] (2,4)-- (2,5);
\draw [line width=2pt,color=purple] (2,5)-- (3,5);
\draw [line width=1pt] (3,5)-- (3,2);
\draw [line width=1pt] (3,2)-- (6,2);
\draw [line width=1pt] (5,1)-- (5,4);
\draw [line width=1pt] (5,4)-- (2,4);
\draw [line width=1pt] (4,5)-- (3,5);
\draw [line width=1pt] (4,5)-- (4,1);
\draw [line width=1pt] (6,3)-- (2,3);
\draw [line width=1pt] (2,3)-- (2,4);
\draw [line width=1pt] (4,1)-- (6,1);
\draw [line width=1pt] (6,1)-- (6,3);
\node at (2.5,4.5) {{\Large 1}};
\node at (2.5,3.5) {{\Large 1}};
\node at (3.5,4.5) {{\Large 1}};
\node at (3.5,3.5) {{\Large 1}};
\node at (3.5,2.5) {{\Large 1}};
\node at (4.5,3.5) {{\Large 1}};
\node at (4.5,2.5) {{\Large 1}};
\node at (4.2,1.1) {{\begin{rotate}{45}\tiny 0\,\text{or}\,1 \end{rotate}}};
\node at (5.2,2.1) {{\begin{rotate}{45}\tiny 0\,\text{or}\,1 \end{rotate}}};
\end{tikzpicture}
\\
&16 & 16 & 4
\end{tabular}
\caption{Total size 3 CPs and DSPPs counted with the weights of \eqref{E:CP_DSPP_weighted_example}. The number under each diagram represents the total number of distinct objects these diagrams represent.}
\label{F:CP_DSPP_figure}
\end{figure}
\end{remark}

As in \cite{HX}, symmetric cylindric partitions can be viewed as weighted DSPPs---that is, for $\delta=(\delta_1, \dots, \delta_h)$ we have
\begin{equation}\label{E:scprewriteasdspp} 
\mathrm{SCP}_{(-\mathrm{rev}(\delta), \delta)}(q)=\mathrm{DSPP}_{\delta}^{(1,2,2,\dots, 2,1)}(q).
\end{equation}

Consider those cylindric partitions with symmetric profile $(-\mathrm{rev}(\delta), \delta)$ which are themselves not symmetric, i.e. the set $\mathrm{CP}_{(-\mathrm{rev}(\delta), \delta)} \setminus \mathrm{SCP}_{(-\mathrm{rev}(\delta), \delta)}$.  One can show that \begin{align*} W_3((-\delta,\delta))=W_1^{(1,2, \dots,2,1)}(\delta) \cup \frac{1}{2}W_2^{(1,2, \dots,2,1)}(\delta) \cup \frac{1}{2}W_2^{(1,2, \dots,2,1)}(\delta)\end{align*} as a disjoint union of multisets, and this leads directly to a generating function with manifestly positive coefficients.
\begin{corollary}
Let $\delta=(\delta_1, \dots, \delta_h)$ be a profile.  Then with notation as in Proposition \ref{P:DSPPproductgeneral}, we have
$$
\sum_{\lambda \in \mathcal{CP}_{(-\mathrm{rev}(\delta), \delta)} \setminus \mathcal{SCP}_{(-\mathrm{rev}(\delta), \delta)}} q^{|\lambda|}=\prod_{\substack{k \in W_1^{(1,2,\dots,2,1)}(\delta) \\ \ell \in W_2^{(1,2,\dots,2,1)}(\delta)}} \frac{1}{(q^k;q^{2h})_{\infty}(q^{\ell};q^{4h})_{\infty}}\left(\prod_{\ell \in W_2^{(1,2,\dots,2,1)}(\delta)} \frac{(-q^{\ell/2};q^{2h})_{\infty}}{(q^{\ell/2};q^{2h})_{\infty}} -1\right).
$$
\end{corollary}

In Section \ref{S:recurrences}, we introduce and prove recurrences for two variable analogues of $\mathrm{CP}_{\delta}(z;q)$ with weighted size,
\begin{equation}\label{E:twovargenfnsdef}
\mathrm{CP}_{\delta}^{{\bf a}}(z;q):=\sum_{\lambda \in \mathcal{CP}_{\delta}} z^{\mathrm{max}(\lambda)}q^{|\lambda|_{{\bf a}}},
\qquad \mathrm{DSPP}_{\delta}^{{\bf a}}(z;q):=\sum_{\lambda \in \mathcal{DSPP}_{\delta}} z^{\mathrm{max}(\lambda)}q^{|\lambda|_{{\bf a}}}.
\end{equation}

Theorem \ref{T:scpwidth2} follows from solving the symmetric cylindric cases of width 4 (or equivalently, the DSPP cases of width 2 and weight $(1,2,1)$), Theorem \ref{T:scpwidth3} from the symmetric cylindric cases of width 6 (or equivalently, the DSPP cases of width 3 and weight $(1,2,2,1)$), and Theorem \ref{C:gollnitz} from the standard weight DSPP cases of width 3.

\section{Systems of recurrences for cylindric partitions and DSPPs}\label{S:recurrences}

To demonstrate Corteel--Welsh's recurrence and our generalizations, consider the following toy example: If $\mathrm{lg}(\lambda)$ denotes the largest part of the integer partition $\lambda$, then one can prove the following recurrence for $P(z;q):=\sum_{\lambda \in \mathcal{P}}z^{\mathrm{lg}(\lambda)}q^{|\lambda|}$,
\begin{equation}\label{E:Partitionrecurrence}
P(z;q)=\sum_{\substack{\mu \in \mathcal{P} \\ m \geq 0}} z^{\mathrm{\lg}(\mu)+m}q^{m + |\mu|}=\frac{P(zq;q)}{1-zq}.
\end{equation}
The proof consists taking a partition $\lambda$ on the left-hand side and removing $\mathrm{lg}(\lambda)$ to create a new partition $\mu$; one then has $\mathrm{lg}(\lambda)=\mathrm{lg}(\mu)+m$ for some $m \geq 0$.  From \eqref{E:Partitionrecurrence}, it is easy to derive Euler's product-sum identity (\cite{A}, Corollary 2.2)
$$
P(z;q)=\sum_{n \geq 0} \frac{z^nq^n}{(q;q)_n}=\frac{1}{(zq;q)_{\infty}}.
$$

If we begin instead with a cylindric partition generated by $\mathrm{CP}_{\delta}(z;q)$ and remove some largest parts in the ``corners'' where they occur, then the resulting cylindric partition has a new profile.  Since subsets of largest parts can be removed in multiple ways, an inclusion-exclusion process leads to the following analogue of \eqref{E:Partitionrecurrence}.
\begin{proposition}[\cite{CW}, Proposition 3.1, reformulated]\label{P:CorteelWelshrecurrence}
Let $\delta=(\delta_1, \dots, \delta_h)$ be a profile, and for convenience define $\delta_{0}:=\delta_h$.  Define $$I_{\delta}:=\{0 \leq j \leq h-1 : (\delta_j,\delta_{j+1})=(1,-1)\}.$$  For a subset $\emptyset \subsetneq J \subseteq I_{\delta}$, define a new profile $\sigma_J(\delta)$ by swapping the signs of $(\delta_j,\delta_{j+1})$ for $j \in J$.  Then
\begin{equation}\label{E:CWrecurrence} {\rm CP}_{\delta}(z;q)=\sum_{\emptyset \subsetneq J \subseteq I_{\delta}} (-1)^{|J|-1}\frac{{\rm CP}_{\sigma_J(\delta)}\left(zq^{|J|};q\right)}{1-zq^{|J|}}.
\end{equation}
\end{proposition}

Similarly, we have the following systems of recurrences for weighted cylindric partitions and weighted DSPPs.  The proofs are simple adjustments of the proof of Proposition 3.1 in \cite{CW}.

\begin{proposition}\label{P:CPrecurrencegeneral}
Let $\delta=(\delta_1, \dots, \delta_h)$ be a profile and let ${\bf a}\in \mathbb{R}_{\geq 0}^h$.  Then with the same notation as in Proposition \ref{P:CorteelWelshrecurrence}, we have
\begin{equation}\label{E:CWrecurrencegeneral} {\rm CP}^{{\bf a}}_{\delta}(z;q)=\sum_{\emptyset \subsetneq J \subseteq I_{\delta}} (-1)^{|J|-1}\frac{{\rm CP}^{{\bf a}}_{\sigma_J(\delta)}\left(zq^{\sum_{j \in J} a_j};q\right)}{1-zq^{\sum_{j \in J} a_j}}.
\end{equation}
\end{proposition}

Note that for DSPPs, the set $\hat{I}_{\delta}$ and maps $\hat{\sigma}_J$ below are subtly different.  
\begin{proposition}\label{P:DSPPrecurrencegeneral}
Let $\delta=(\delta_1, \dots, \delta_h)$ be a profile and let ${\bf a} \in \mathbb{R}_{\geq 0}^{h+1}$.  For convenience, define $\delta_{h+1}:=-\delta_h$ and $\delta_0:=-\delta_1$.  Define $$\hat{I}_{\delta}:=\{0 \leq j \leq h : (\delta_j,\delta_{j+1})=(1,-1)\}.$$  For a subset $\emptyset \subsetneq J \subseteq I_{\delta}$, define a new profile $\hat{\sigma}_J(\delta)$ by swapping the signs of $(\delta_j,\delta_{j+1})$ for $j \in J$.  Then
\begin{equation}\label{E:DSPPrecurrencegeneral} {\rm DSPP}_{\delta}^{{\bf a}}(z;q)=\sum_{\emptyset \subsetneq J \subseteq \hat{I}_{\delta}} (-1)^{|J|-1}\frac{{\rm DSPP}^{{\bf a}}_{\hat{\sigma}_J(\delta)}\left(zq^{\sum_{j \in J}a_j};q\right)}{1-zq^{\sum_{j \in J}a_j}}.
\end{equation}
\end{proposition}

\begin{remark}
For cylindric partitions, the maps $\sigma_J$ permute the entries of profiles, so the rank---the number of $(-1)$'s---is preserved.  Hence, there is a system of $\binom{h}{k}$ recurrences \eqref{E:CWrecurrencegeneral} for width $h$, rank $k$ profiles $\delta.$

By contrast, for DSPPs the maps $\sigma_J$ act transitively on all profiles of width $h$.  Thus, in general there is a single system of $2^h$ recurrences \eqref{E:DSPPrecurrencegeneral}.  However, for some weights, like those corresponding to symmetric cylindric partitions, the number of equations can be reduced through certain symmetries.
\end{remark}

Following \cite{CW}, our analysis will simplify if the denominators in \eqref{E:CWrecurrencegeneral} and \eqref{E:DSPPrecurrencegeneral} are removed by defining
$$
G_{\delta}(z):=(zq;q)_{\infty}\mathrm{CP}^{{\bf a}}_{\delta}(z;q).
$$
(We will suppress the weight ${\bf a}$ and the second variable $q$ in $G_{\delta}$ to save space when they are clear from context.)  Then \eqref{E:CWrecurrencegeneral} becomes
\begin{equation}\label{E:adjustedCWrecurrencegeneral}
    G_{\delta}(z)=\sum_{\emptyset \subsetneq J \subseteq I_{\delta}}(-1)^{|J|-1} (zq;q)_{\sum_{j \in J} a_j-1} G_{\sigma_J(\delta)}\left(zq^{\sum_{j \in J}a_j}\right).
\end{equation}
Likewise, with $a_j \geq 1$ for all $j$ and
$$
H_{\delta}(z):=(zq;q)_{\infty}\mathrm{DSPP}^{{\bf a}}_{\delta}(z;q),
$$
\eqref{E:DSPPrecurrencegeneral} becomes
\begin{equation}\label{E:adjustedDSPPrecurrencegeneral}
    H_{\delta}(z)=\sum_{\emptyset \subsetneq J \subseteq I_{\delta}}(-1)^{|J|-1} (zq;q)_{\sum_{j \in J} a_j-1} H_{\sigma_J(\delta)}\left(zq^{\sum_{j \in J}a_j}\right).
\end{equation}

\section{Proofs of Theorems~\ref{T:scpwidth2}, \ref{T:scpwidth3}, and \ref{C:gollnitz}}\label{S:proofs}

Although the presented results can be performed by hand, to avoid error-prone and tedious calculations, we mainly use two symbolic computation implementations in Mathematica to carry out our calculations: the package \texttt{qFunctions} by Ablinger and the second author \cite{qFuncs}, and the package \texttt{HolonomicFunctions} by Koutchan \cite{HoloFuncs}. These implementations are distributed through the RISC (Research Institute for Symbolic Computation, Johannes Kepler University, Linz) openly for researchers to benefit from. These implementations and more can be downloaded through \url{https://risc.jku.at/software/}.

Koutchan's \texttt{HolonomicFunctions} \cite{HoloFuncs} implementation is  well established and offers state of the art technology in holonomic functions research (in our context, functions that satisfy linear $q$-recurrence relations). In particular, this package includes functionality that automatically uncouples a coupled system of recurrences and it also has the creative telescoping algorithm to automatically produce, proof, and certify recurrence relations for given closed hypergeometric formulas. 

The \texttt{qFunctions} package \cite{qFuncs} by Ablinger and the second author offers many tools to symbolically manipulate (such as by doing substitutions) linear functional relations of $q$-holonomic functions and to symbolically create the $q$-difference relations that Corteel--Welsh \cite{CW} originally presented. We include a short Appendix to list the new functionality (that is soon to be added to the main \texttt{qFunctions} release) which implements new functions related to this paper.

It might be possible to prove some of these identities using hypergeometric means directly. However, we would like to present proofs solely through coupled system of $q$-difference equations, which were introduced in Section~\ref{S:recurrences}. This is to emphasize the technique of how one can go about reducing the coupled system and under fortunate circumstances guess a sum representation of the generating functions for the cylindric partitions and other objects. These proofs are in the spirit of \cite{CW} and \cite{CDU}.

\subsection{Proof of Theorem~\ref{T:scpwidth2} and related results:}\label{SubS:SCP4} Earlier we defined the symmetric cylindric partitions with regards to DSPPs  \eqref{E:scprewriteasdspp}. We extract  Theorems~\ref{T:scpwidth2} and \ref{T:scpwidth3} through the study of the generating functions of the number of symmetric cylindric partitions width 4 and 6, respectively. 

\begin{lemma}\label{L:width2zqSCPproducts} For the profiles with width 4, we have  
\begin{align*}
     \mathrm{SCP}_{(-1,-1,1,1)}(z;q) &= \frac{(-z q^4 , z q^2;q^4)_\infty}{(zq;q)_\infty},\\
    \mathrm{SCP}_{(1,-1,1,-1)}(z;q) &= \frac{(-z q^3 , z q;q^4)_\infty}{(zq;q)_\infty},\\
    \mathrm{SCP}_{(-1,1,-1,1)}(z;q) &= \frac{(-z q , z q^3;q^4)_\infty}{(zq;q)_\infty}.
\end{align*}
\end{lemma}
\begin{proof}
We get the following coupled system of $q$-difference equations when we apply  \eqref{E:adjustedDSPPrecurrencegeneral} to the width 2 profiles with $a = (1,2,1)$:
\begin{align}
    \label{E:H11}  H_{(1,1)}(z) &= H_{(1,-1)}(z q),\\
    \label{E:H1-1} H_{(1,-1)}(z) &= (1- z q) H_{(-1,1)}(z q^2),\\
    \label{E:H-11} H_{(-1,1)}(z) &= 2 H_{(1,1)}(z q) - (1-z q) H_{(1,-1)}(z q^2), 
\end{align} where $H_\delta(z) := (z q;q)_\infty \SCP_{(-\mathrm{rev}(\delta),\delta)}(z)$. 

Recall that $\SCP_{(1,1,-1,-1)}(z;q)= \SCP_{(-1,-1,1,1)}(z;q)$, hence, we do not need to consider $H_{(-1,-1)}(z)$ as a separate entity. 

We can uncouple these $q$-difference equations for the $H_{(-1,1)}(z)$ function by substituting \eqref{E:H11} followed by substituting \eqref{E:H1-1} in \ref{E:H-11}. This yields \begin{align}
    \nonumber H_{(-1,1)}(z) &= 2(1-z q^3)H_{(-1,1)}(z q^4) - (1+zq)(1-zq^3)H_{(-1,1)}(zq^4)\\\label{E:uncoupH-11} &= (1+zq)(1-zq^3)H_{(-1,1)}(z q^4).
\end{align}
Iterating \eqref{E:uncoupH-11} shows that \begin{equation}\label{E:H-11prod}
    H_{(-1,1)}(z) = (-z q , z q^3;q^4)_\infty.
\end{equation} Substituting \eqref{E:H-11prod} in \eqref{E:H1-1} shows that $H_{(1,-1)}(z) = (-z q^3 , z q;q^4)_\infty$ and finally substituting this product in \eqref{E:H11} proves that $H_{(1,1)}(z) = (-z q^4 , z q^2;q^4)_\infty.$ Writing the definitions of $H_\delta(z)$s in their respective product representations proves Lemma~\ref{L:width2zqSCPproducts}.
\end{proof}

An alternate proof of Lemma~\ref{L:width2zqSCPproducts}, which will also lead to the proof of Theorem~\ref{T:scpwidth2}, can be given through uncoupling the recurrence system. Let \[H_\delta(z;q) := \sum_{n\geq 0} h_\delta(n) z^n.\] Then the $q$-difference equation system that the $H_\delta(z)$ functions satisfy is equivalent to the following coupled system of recurrences.
\begin{align*}
    h_{(1,1)}(n) &= q^n h_{(1,-1)}(n),\\
    h_{(1,-1)}(n) &= q^{2n} h_{(-1,1)}(n) - q^{2n-1} h_{(-1,1)}(n-1),\\
    h_{(-1,1)}(n) &= 2q^{n} h_{(1,1)}(n) -q^{2n} h_{(1,-1)}(n) + q^{2n-1}(n-1).
\end{align*} 
The initial conditions $h_\delta(0)=1$ and $h_\delta(m)=1$ for all negative integer $m$, defines these sequences uniquely. This system of recurrences can be uncoupled and we see that these coefficients satisfy the second order recurrences
\begin{align}
    \label{E:uncoup_h11_rec}(1 - q^{4n+8}) h_{(1,1)}(n+2) &= -q^{4n+6} (1 - q^2) h_{(1,1)}(n+1) - q^{4n+6}h_{(1,1)}(n),\\
    \label{E:uncoup_h1-1_rec}(1 - q^{4n+8}) h_{(1,-1)}(n+2) &= -q^{4n+5} (1 - q^2) h_{(1,-1)}(n+1) - q^{4n+4}h_{(1,-1)}(n),\\
    \label{E:uncoup_h-11_rec}(1 - q^{4n+8}) h_{(-1,1)}(n+2) &= -q^{4n+5} (1 - q^2) h_{(-1,1)}(n+1) + q^{4n+4}h_{(-1,1)}(n).
\end{align} 

We can switch from these uncoupled $q$-recurrence relations to the equivalent $q$-difference equations. This yields another proof of Lemma~\ref{L:width2zqSCPproducts} after confirming that the $q$-difference equations one get from \eqref{E:uncoup_h11_rec}-\eqref{E:uncoup_h-11_rec} are homogeneous two term relations (one of which is \eqref{E:uncoupH-11}).

However, one can also look at the recurrences \eqref{E:uncoup_h11_rec}-\eqref{E:uncoup_h-11_rec} and some initial conditions to guess (and later prove) a closed formula for these sequences. In our cases we have \begin{align}
    \label{E:h11_formula}h_{(1,1)}(n) = (-1)^{\lceil n/2\rceil} q^{n(n+1)} \frac{(q^2;q^4)_{\lceil n/2\rceil} (-q^4;q^4)_{\lfloor n/2\rfloor} }{(q^4;q^4)_n},\\
    \label{E:h1-1_formula}h_{(1,-1)}(n) = (-1)^{\lceil n/2\rceil} q^{n^2} \frac{(q^2;q^4)_{\lceil n/2\rceil} (-q^4;q^4)_{\lfloor n/2\rfloor} }{(q^4;q^4)_n},\intertext{and}
    \label{E:h-11_formula}h_{(-1,1)}(n) = (-1)^{\lfloor n/2\rfloor} q^{n^2} \frac{(q^2;q^4)_{\lceil n/2\rceil} (-q^4;q^4)_{\lfloor n/2\rfloor} }{(q^4;q^4)_n},
\end{align} where $\lceil\cdot \rceil$ and $\lfloor\cdot \rfloor$ are the classical ceil and floor functions. The right-hand side expressions satisfy the same initial conditions as the expressions on the left; they vanish for every negative integer $n$, and when $n=0$ they become 1. We can easily show that the expressions \eqref{E:h11_formula}-\eqref{E:h-11_formula} satisfy \eqref{E:uncoup_h11_rec}-\eqref{E:uncoup_h-11_rec}, respectively. For example, letting $n\mapsto 2n$, in \eqref{E:h-11_formula}, looking at the difference of the two sides of \eqref{E:uncoup_h11_rec} we see that everything vanishes:
\begin{align*}
    (1 -& q^{8n+8}) (-1)^{n+1} q^{(2n+2)(2n+3)} \frac{(q^2;q^4)_{n+1} (-q^4;q^4)_{n+1}}{(q^4;q^4)_{2n+2}}\\&+q^{8n+6} (1 - q^2) (-1)^{n+1} q^{(2n+1)(2n+2)} \frac{(q^2;q^4)_{n+1}(-q^4;q^4)_n}{(q^4;q^4)_{2n+1}}+q^{8n+6}(-1)^n q^{2n(2n+1)} \frac{(q^2, -q^4;q^4)_n}{(q^4;q^4)_{2n}} \\[-2ex]\\
    =& (-1)^{n+1}  q^{(2n+1)(2n+2)} \frac{(q^2;q^4)_{n+1}(-q^4;q^4)_n}{(q^4;q^4)_{2n+1}} (q^{4n+4}(1+q^{4n+4})+q^{8n+6}(1 - q^2))\\
    &+q^{8n+6}(-1)^n q^{2n(2n+1)} \frac{(q^2, -q^4;q^4)_n}{(q^4;q^4)_{2n}} \\[-2ex]\\
    =&q^{8n+6}(1+q^{4n+2}) (-1)^{n+1}  q^{2n(2n+1)} \frac{(q^2;q^4)_{n+1}(-q^4;q^4)_n}{(q^4;q^4)_{2n+1}}+q^{8n+6}(-1)^n q^{2n(2n+1)} \frac{(q^2, -q^4;q^4)_n}{(q^4;q^4)_{2n}}=0.
\end{align*} This proves that the right-side expression in \eqref{E:h11_formula} satisfies \eqref{E:uncoup_h11_rec}. By carrying the same calculations for $h_{(1,1)}(2n+1)$ and for the other two sequences we prove that \eqref{E:h11_formula}-\eqref{E:h-11_formula} satisfy \eqref{E:uncoup_h11_rec}-\eqref{E:uncoup_h-11_rec}, respectively. Lastly, we check the initial conditions, and that finishes the proof the equalities in \eqref{E:h11_formula}-\eqref{E:h-11_formula}.

Rewriting $H_{(-1,1)}(z)$ using the formula for $h_{(-1,1)}(n)$ in \eqref{E:H-11prod} yields \begin{equation}\label{E:H_sum_product}
    H_{(-1,1)}(z) = \sum_{n\geq 0} h_{(-1,1)}(n)z^n = \sum_{n\geq 0} (-1)^{\lfloor n/2\rfloor} q^{n^2} \frac{(q^2;q^4)_{\lceil n/2\rceil} (-q^4;q^4)_{\lfloor n/2\rfloor} }{(q^4;q^4)_n}z^n = (-zq,zq^3;q^4)_\infty.
\end{equation} Splitting the even and odd cases of the summand in \eqref{E:H_sum_product} and regrouping the terms proves Theorem~\ref{T:scpwidth2}. Moreover, the analogous calculations for $H_{(1,1)}(z)$ and $H_{(1,-1)}(z)$ also lead to Theorem~\ref{T:scpwidth2}.

It should be noted that one can give a shorter proof to Theorem~\ref{T:scpwidth2} without resorting to the coupled system of symmetric cylindric partitions of width 4. For example, one can use the $q$-binomial theorem to expand $(zq;q^4)_\infty$ and $(-zq^3;q^4)_\infty$ separately, multiply these series expansions, collect and compare the terms of $z^N$ of this expansion with the $z^N$ term of the left-side sum in Theorem~\ref{T:scpwidth2}. Then one can find the recurrences for these compared terms with $q$-Zeilberger algorithm and finish the proof. Another alternative to recurrences after the comparison is to observe that these the convolution sums one gets from the $q$-binomial theorems are equivalent to formulas (21) and (22) in \cite{Paule} (corresponding to even and odd exponents of $z$ in the expansions). 

Setting $z=q$ followed by $q^2\mapsto q$ in Theorem~\ref{T:scpwidth2} gives the following corollary.

\begin{corollary}\label{C:weightedPartitionSumGF}
\begin{equation}
    \label{E:scpwidth2_zq} \sum_{n\geq 0}(-1)^{n} q^{2n^2+n} \frac{(q;q^2)_{n+1}(-q^2;q^2)_{n}}{(q^2;q^2)_{2n+1}} = (q, -q^2;q^2)_\infty.
\end{equation}
\end{corollary}

The identity \eqref{E:scpwidth2_zq} can also be proven through combinatorial arguments. The right-side product is the generating function for the number of partitions into distinct parts with the extra weight $-1$ to the number of odd parts in the partition. 

 One can interpret the left-hand side summand as the generating function of partitions into $2n$ or $2n+1$ distinct parts, where the partitions are counted with the extra weight $-1$ to the number of odd parts. This interpretation can be seen starting from the front factor $(-1)^{n} q^{2n^2+n}$. The term \[2n^2+n = 1+2+3+\dots+(2n-1)+2n\] is the generating function for partitions into exactly $2n$ consecutive parts. Furthermore, half of these $2n$ parts are odd and the $(-1)^n$ reflects the weight on the number of odd parts.  We interpret the $(-q^2;q^2)_n$ as the generating function of Young diagrams with columns even length $\leq 2n$ where each column length only appears once. Similarly $(q;q^2)_{n+1}$ is the generating function of Young diagrams with columns odd length $\leq 2n+1$ where each column length only appears once, where in this case we also have the extra weight $-1$ to the number of columns. Putting together these three terms, we can see that \[(-1)^{n} q^{2n^2+n}(q;q^2)_{n+1}(-q^2;q^2)_n\] is the generating function of partitions into $2n$ or $2n+1$ parts, where the smallest part is at most 2 if there are $2n$ parts, it is exactly 1 if there are $2n+1$ parts, and the difference between consecutive parts is at most 2. Moreover, each insertion of an odd length column in the original Young diagram of the partition with size $2n^2+n$ changes the number of odd parts by one, and each of these column also comes with a $-1$ weight itself that accounts for this change. Hence, combining the $-1$ weights we see that the exponent of $-1$ the number of odd parts in the outcome partition. Finally, the term $1/(q^2;q^2)_{2n+1}$ is the generating function of Young diagrams with even length $\leq 2n+1$ columns where each column appears even number of times. With this final observation, we can now interpret the summand \[(-1)^{n} q^{2n^2+n} \frac{(q;q^2)_{n+1}(-q^2;q^2)_{n}}{(q^2;q^2)_{2n+1}}\] as the generating function for partitions into $2n$ or $2n+1$ distinct parts with an extra alternating weight on the number of odd parts in the partitions. Summing over all possible $n$ we see that the left-hand side sum counts the same partitions as the right-hand side product.
 
 \subsection{Proof of Theorem~\ref{T:scpwidth3}:} Similar to the previous Subsection~\ref{SubS:SCP4}, we study the generating functions width 6 profiled symmetric cylindric partitions. Applying the normalized weighted DSPP recurrence \eqref{E:adjustedDSPPrecurrencegeneral} to the width 3 profiles with $a = (1,2,2,1)$ gives the coupled $q$-difference equation system  
\begin{align}
    \label{E:H111}  H_{\alpha}(z) &= H_{\beta}(z q),\\
    \label{E:H11-1} H_{\beta}(z) &= (1- z q) H_{\gamma}(z q^2),\\
    \label{E:H1-11} H_{\gamma}(z) &= H_{\beta}(z q) - (1-z q)(1-z q^2) H_{\gamma}(z q^3)+(1-zq)H_{\sigma}(zq^2),\\
    \label{E:H-111} H_\sigma(z) &= H_\alpha(zq) - (1-zq)H_\beta(zq^2) + H_{\gamma}(zq),
\end{align} where $$\alpha:=(1,1,1), \quad \beta:=(1,1,-1), \quad \gamma:=(1,-1,1), \quad \sigma:=(-1,1,1).$$ and \begin{equation}\label{E:H_def_SCP3}H_\delta(z) := (z q;q)_\infty \SCP_{(-\mathrm{rev}(\delta),\delta)}(z) := \sum_{n\geq0} h_{\delta}(n) z^n \end{equation} for a profile $\delta$. 

We can once again turn these $q$-difference equations to the equivalent $q$-recurrence equations. Once these recurrences are uncoupled, we get that these sequences satisfy the recurrences
\begin{align}
    \nonumber (1-q^{3n+9})(1+q^{3n+5}-&q^{3n+6})h_\alpha(n+3)= -q^{3n+7}(1-q^3-q^{3n+11}-q^{6n+16}+q^{6n+17}) h_\alpha(n+2)\\\label{E:H111_rec} &- q^{3n+9}(1+q^{3n+5}+q^{3n+7}+q^{3n+8}-q^{3n+9}+q^{6n+12}-q^{6n+14})h_\alpha(n+1)\\ \nonumber &+q^{6n+12}(1+q^{3n+8}-q^{3n+9})h_\alpha(n),\\[-2.2ex]\nonumber\\
    \nonumber (1-q^{3n+9})(1+q^{3n+5}-&q^{3n+6})h_\beta(n+3)= -q^{3n+6}(1-q^3-q^{3n+11}-q^{6n+16}+q^{6n+17}) h_\beta(n+2)\\\label{E:H11-1_rec} &- q^{3n+7}(1+q^{3n+5}+q^{3n+7}+q^{3n+8}-q^{3n+9}+q^{6n+12}-q^{6n+14})h_\beta(n+1)\\ \nonumber &+q^{6n+9}(1+q^{3n+8}-q^{3n+9})h_\beta(n),\\[-2.2ex]\nonumber\\
    \label{E:H1-11_rec}(1-q^{3n+9})h_\gamma(n+3)&= q^{6n+15}h_\gamma(n+2) - q^{3n+6}(1+q^{3n+4}+q^{3n+8})h_\gamma(n+1)+q^{6n+9}h_\gamma(n),\\
     \label{E:H-111_rec}(1-q^{3n+9})h_\sigma(n+3)&= q^{6n+13}h_\sigma(n+2) - q^{3n+8}(1+q^{3n+2}+q^{3n+4})h_\sigma(n+1)+q^{6n+9}h_\sigma(n).
\end{align} 
The initial conditions $h_\delta(0)=1$ and $h_\delta(m)=0$ for all negative $m$ defines these sequences uniquely. 

Looking at some initial values of the $h_\gamma(n)$ sequence, we can identify that \begin{align}
    \label{E:H_gamma_formula} h_\gamma(n) &=  \sum_{m\geq 0} (-1)^{m} q^{3{n+1\choose 2} -3m(m+1)} \frac{(-q,-q^5;q^6)_m}{(q^6;q^6)_m(q^3;q^3)_{n-2m}}.
\intertext{Using $q$-Zeilberger algorithm (or the creative telescoping algorithm), we can automatically prove that the right-hand side of \eqref{E:H_gamma_formula} satisfies \eqref{E:H1-11_rec}. In the same way, using the relations \eqref{E:H111}-\eqref{E:H-111}, we can prove}
     \label{E:H_beta_formula} h_\beta(n) &=  \sum_{m\geq 0} (-1)^{m+1} q^{3{n-1\choose 2}+2n -3m(m+1)-1} \frac{(-q,-q^5;q^6)_m}{(q^6;q^6)_m(q^3;q^3)_{n-2m}} (1-q^{3n+1}+q^{3n-6m}),\\
    \label{E:H_alpha_formula} h_\alpha(n) &=  \sum_{m\geq 0} (-1)^{m+1} q^{3{n+1\choose 2} -3m(m+1)-1} \frac{(-q,-q^5;q^6)_m}{(q^6;q^6)_m(q^3;q^3)_{n-2m}} (1-q^{3n+1}+q^{3n-6m}),
     \intertext{and}
    \label{E:H_sigma_formula} h_\sigma(n) &=  \sum_{m\geq 0} (-1)^{m+1} q^{3{n+1\choose 2}-2n-3m(m+1)-1} \frac{(-q,-q^5;q^6)_m}{(q^6;q^6)_m(q^3;q^3)_{n-2m}}\\  \nonumber &\hspace{2cm}\times(1-q^{3 n+1}-q^{3 n-2}-q^{3 n-6 m}-q^{3n-6 m-3}+q^{6n-6 m-2}+q^{6n-12 m-3})
\end{align} in this order. 

Notice that the formulas in \eqref{E:H_gamma_formula}-\eqref{E:H_sigma_formula} are the summands of the left-hand side sums of Theorem~\ref{T:scpwidth3}. Combining the definition of each  $H_\delta(z)$ \eqref{E:H_def_SCP3}, and Proposition~\ref{P:DSPPproductgeneral} (for $z=1$) for these profiles and weights finishes the proof of Theorem~\ref{T:scpwidth3}.

\subsection{New proofs of Little G\"{o}llnitz and G\"{o}llnitz--Gordon Theorems}
We use the system of recurrences for standard weight, width 3 DSPPs.  For standard weight, reversing the order of the integer partitions in a DSPP leads to the identity $\mathrm{DSPP}_{\delta}(z;q)=\mathrm{DSPP}_{-\mathrm{rev}(\delta)}(z;q)$, thus we need only write the system of recurrences using profiles 
$$\alpha:=(1,1,1), \quad \beta:=(1,1,-1), \quad \gamma:=(1,-1,1), \quad \sigma:=(-1,1,1).$$  From equation \ref{E:adjustedDSPPrecurrencegeneral}, we have the system
\begin{align*}
    H_{\alpha}(z)&=H_{\beta}(zq), \\
    H_{\beta}(z)&=H_{\gamma}(zq^2), \\
    H_{\gamma}(z)&=H_{\beta}(zq)+H_{\delta}(zq)-(1-zq)H_{\gamma}(zq^2),\\
    H_{\sigma}(z)&=H_{\alpha}(zq)+H_{\gamma}(zq)-(1-zq)H_{\beta}(zq^2),
\end{align*}
whence
$$
H_{\sigma}(z)=H_{\gamma}(zq)+zqH_{\gamma}(zq^3),
$$
and
$$
H_{\gamma}(z)=(1+zq)H_{\gamma}(zq^2)+zq^2H_{\gamma}(zq^4).
$$
If we now write $H_{\gamma}(z)=\sum_{n \geq 0} h_{\gamma}(n)z^n,$ where the $h_n$ are rational functions in $q$ and $h_{\gamma}(0)=H_{\gamma}(0)=1,$ then the above recurrence gives
$$
h_{\gamma}(n)=q^{2n}h_{\gamma}(n)+(q^{2n-1}+q^{4n-2})h_{\gamma}(n-1).
$$
By iterating, we obtain \[h_{\gamma}(n)=\frac{(-q;q^2)_n}{(q^2;q^2)_n}q^{n^2},\] therefore,
$$
\sum_{n \geq 0} \frac{(-q;q^2)_n}{(q^2;q^2)_n}q^{n^2}=H_{(1,-1,1)}(1)=(q;q)_{\infty}\mathrm{DSPP}_{(1,-1,1)}(q)=\frac{(q;q)_{\infty}}{(q,q^2,q^3,q^4;q^4)_{\infty}(q^2,q^4,q^7;q^8)_{\infty}}.
$$
Upon cancelling common terms in the product, we obtain the third identity in Theorem \ref{C:gollnitz}.  The other identities can be proved using the expression for $h_{\gamma}(n)$ and the relations above.

\section{Applications to Schmidt-type identities: proof of Theorem \ref{T:RefinedAndrewsPaule}}\label{S:Applications}

In this section we prove Theorem \ref{T:RefinedAndrewsPaule}, beginning with identity \eqref{E:refinedschmidtdiamond}.  It is clear from the definitions that diamond partitions can be thought of as weighted DSPPs.  In particular,
$$
\sum_{\lambda \in \diamondsuit}z^{\lambda_1}q^{\lambda_1+\lambda_4+\lambda_7 + \dots}=\mathrm{DSPP}_{(1,-1)}^{(0,1,0)}(z;q).
$$
Note that Proposition \ref{P:DSPPproductgeneral} then immediately implies \eqref{E:schmidtdiamond}, as
$$
\sum_{\lambda \in \diamondsuit}q^{\lambda_1+\lambda_4+\lambda_7 + \dots}=\frac{1}{(q;q)_{\infty}^3(q;q^2)_{\infty}}=\frac{(-q;q)_{\infty}}{(q;q)_{\infty}^3}.
$$
To prove our refinement, we apply the recurrences in Proposition \ref{P:DSPPrecurrencegeneral} to obtain the system,
\begin{align*}
&(1-z)\mathrm{DSPP}_{(1,1)}^{(0,1,0)}(z;q)=\mathrm{DSPP}_{(1,-1)}^{(0,1,0)}(z),\\
&(1-z)\mathrm{DSPP}_{(-1,-1)}^{(0,1,0)}(z;q)=\mathrm{DSPP}_{(1,-1)}^{(0,1,0)}(z),\\
    &(1-zq)\mathrm{DSPP}_{(1,-1)}^{(0,1,0)}(z;q)=\mathrm{DSPP}_{(-1,1)}^{(0,1,0)}(zq), \\
    &(1-z)\mathrm{DSPP}_{(-1,1)}^{(0,1,0)}(z;q)=\mathrm{DSPP}_{(1,1)}^{(0,1,0)}(z)+\mathrm{DSPP}_{(-1,-1)}^{(0,1,0)}(z)-\mathrm{DSPP}_{(1,-1)}^{(0,1,0)}(z).
\end{align*}
It is then easy to show that
$$
\mathrm{DSPP}_{(1,-1)}^{(0,1,0)}(z;q)=\frac{1+zq}{(1-zq)^3} \mathrm{DSPP}_{(1,-1)}^{(0,1,0)}(zq;q)=\frac{(-zq;q)_{\infty}}{(zq;q)_{\infty}^3},
$$
proving \eqref{E:refinedschmidtdiamond}.

Identities \eqref{E:refinedschmidtunrestricted1} and \eqref{E:refinedschmidtunrestricted2} are proved similarly by showing that
$$
\sum_{\lambda \in \mathcal{P}}z^{\lambda_1}q^{\lambda_1+\lambda_3+\dots}=\mathrm{CP}_{(-1,1)}^{(0,1)}(z;q), \quad\text{and}\quad \sum_{\lambda \in \mathcal{P}}z^{\lambda_1}q^{\lambda_2+\lambda_4+\dots}=\mathrm{CP}_{(1,-1)}^{(0,1)}(z;q),
$$
with the aid of the recurrence in Proposition \ref{P:CPrecurrencegeneral}. 

For identities \eqref{E:refinedschmidtdistinct1} and \eqref{E:refinedschmidtdistinct2}, we introduce {\it cylindric partitions into distinct parts}, which are defined so that the inequalities along rows and columns in a cylindric partition are strict.  Define the generating function
$
\mathrm{DCP}_{\delta}^{{\bf a}}(z;q)
$ analogously.  Then
$$
\sum_{\lambda \in \mathcal{D}}z^{\lambda_1}q^{\lambda_1+\lambda_3+\dots}=\mathrm{DCP}_{(-1,1)}^{(0,1)}(z;q), \qquad \sum_{\lambda \in \mathcal{D}}z^{\lambda_1}q^{\lambda_2+\lambda_4+\dots}=\mathrm{DCP}_{(1,-1)}^{(0,1)}(z;q).
$$
In this case, the recurrence is slightly different from Proposition \ref{P:CPrecurrencegeneral} but is proved as in \cite{CW}.
\begin{proposition}
With notation as in Proposition \ref{P:CPrecurrencegeneral}, we have
$$
\mathrm{DCP}_{\delta}^{{\bf a}}(z;q)=1+\sum_{\emptyset \subsetneq J \subseteq I_{\delta}} (-1)^{|J|-1}zq^{\sum_{j \in J} a_j}\frac{\mathrm{DCP}^{{\bf a}}_{\sigma_J(\delta)}\left(zq^{\sum_{j \in J}a_j};q\right)}{1-zq^{\sum_{j \in J} a_j}}.
$$
\end{proposition}
Note that the $q$-difference equations above are inhomogeneous, in contrast to those for cylindric partitions and DSPPs.

Thus, we have the system
\begin{align*}
    \mathrm{DCP}_{(-1,1)}^{(0,1)}(z;q)&=1+\frac{z}{1-z}\mathrm{DCP}_{(1,-1)}^{(0,1)}(z;q) \\
    \mathrm{DCP}_{(1,-1)}^{(0,1)}(z;q)&=1+\frac{zq}{1-zq}\mathrm{DCP}_{(-1,1)}^{(0,1)}(zq;q),
\end{align*}
which implies
$$
\mathrm{DCP}_{(1,-1)}^{(0,1)}(z;q)=1+\frac{zq}{1-zq}+\frac{z^2q^2}{(1-zq)^2}\mathrm{DCP}_{(1,-1)}^{(0,1)}(zq;q)=\frac{1}{1-zq}+\frac{z^2q^2}{(1-zq)^2}\mathrm{DCP}_{(1,-1)}^{(0,1)}(zq;q),
$$
which one may iterate to get
\begin{equation}\label{E:DCPlastLine}
\mathrm{DCP}_{(1,-1)}^{(0,1)}(z;q)=\sum_{n \geq 0} \frac{z^{2n}q^{n(n+1)}}{(zq;q)_{n}(zq;q)_{n+1}},    
\end{equation}
which gives \eqref{E:refinedschmidtdistinct1}.  Identity \eqref{E:refinedschmidtdistinct2} can also be derived in a similar fashion by studying $\mathrm{DCP}_{(-1,1)}^{(0,1)}(z;q)$.

\section{Conclusion and Outlook}\label{S:Outlook}

Corteel and Welsh \cite{CW} developed the idea of using coupled systems of $q$-difference equations for cylindric partitions to prove sum-product identities. This paper expands their original idea in multiple ways. We apply Corteel and Welsh's argument to symmetric cylindric partitions and skew double shifted plane partitions. The introduction of non-trivial weights then allows for a bigger search space. We do not have a clear vision of which profile families (those having the same width and same rank) or which group of weights would yield an essentially solvable $q$-difference equation system.

New combinatorial questions arise from the correspondences noted in Remark \ref{R:differentproducts}. Can we define equivalence classes of cylindric partitions (such as for the ones in \eqref{E:CPs1} and \eqref{E:CPs2}) or DSPPs? Are there direct bijections between these sets of objects?  Can finite versions of these sum-product identities be proved?

 In the alternate proof of Corollary~\ref{C:weightedPartitionSumGF}, we showed that when $z=q$ both sides of the identity in the Theorem \ref{T:scpwidth2} can be interpreted as a generating function for a weighted count of partitions into distinct parts. It would be interesting to see partition theoretic interpretations of Theorem~\ref{T:scpwidth2} for other specializations of $z$.

We should also note that Lemma~\ref{L:width2zqSCPproducts} and our other experiments suggest that the generating functions in many cases actually yield bivariate products in $q$ and $z$ rather than the univariate products coming from Propositions~\ref{P:borodin}, \ref{P:CPgeneralproduct} and \ref{P:DSPPproductgeneral}. Similarly, in Section \ref{S:Applications} we discussed cylindric partitions into distinct parts and observed through Theorem~\ref{T:AndrewsPaule} that for some width 2, rank 1 profiles the univariate  generating functions (when $z=1$) turn into infinite products. It is worth investigating whether we get bivariate products in other cases as well.

The authors plan to return to some of these questions in follow-up work, guided with the help of the new symbolic computation implementations tailor-made for this purpose.

\section*{Appendix A: modularity of $\mathrm{CP}_{\delta}(q)$}

We prove that the multiset $W_3(\delta)$ in Proposition \ref{P:borodin} contains every $k<h$ as many times as it contains $h-k.$  We say that $W_3(\delta)$ is {\it $k$-balanced} if the above property is true for $k$.

We begin by showing that $W_3(\delta)$ is 1-balanced.  It is easy to see that $W_3(-\delta)$ contains $h-k$ exactly as many times as $W_3(\delta)$ contains $k$, so $W_3(-\delta)$ is $k$-balanced if and only if $W_3(\delta)$ is.  Hence, we may assume without loss of generality that $\delta_1=1.$  Thus, the multiplicity of 1 in $W_3(\delta)$ is
$$
\#\{(\delta_j,\delta_{j+1})=(1,-1)\},
$$
and the multiplicity of $h-1$ in $W_3(\delta)$ is
$$
\#\{(\delta_j,\delta_{j+1})=(-1,1)\}+1_{\delta_h=-1}.
$$
If the number of sign changes in the sequence $\delta$ is even, then $\delta_h=1$ and the two multiplicities above are equal.  If the number of sign changes is odd, then $\delta_h=-1$ and furthermore $$
\#\{(\delta_j,\delta_{j+1})=(1,-1)\}
=1+
\#\{(\delta_j,\delta_{j+1})=(-1,1)\},
$$
as required.  Thus, $W_3(\delta)$ is 1-balanced for all profiles $\delta$.

Now let $1<k<h-1$ and partition $\delta$ into the following sub-profiles,
\begin{align*}
    &(\delta_1, \delta_{1+k}, \dots, \delta_{1+q_1k}) \\
    &(\delta_2, \delta_{2+k}, \dots, \delta_{2+q_2k}) \\
    &\hspace{10mm} \vdots \\
    &(\delta_{k-1}, \delta_{2k-1}, \dots, \delta_{(q_{k-1}+1)k-1}),
\end{align*}
where each $q_r$ satisfies $q_rk+r\leq h<(q_r+1)k+r.$  The multiplicity of $k$ in $W_3(\delta)$ is
$$
\#\{(\delta_j,\delta_{j+k})=(1,-1)\} + \#\{(\delta_j,\delta_{j+h-k})=(-1,1)\},
$$
and the multiplicity of $h-k$ in $W_3(\delta)$ is
$$
\#\{(\delta_j,\delta_{j+k})=(-1,1)\} + \#\{(\delta_j,\delta_{j+h-k})=(1,-1)\}.
$$
Here, the $k$ pairs $\{(\delta_j,\delta_{j+h-k})\}_{1\leq j \leq k}$ occurring in $\delta$ are evenly distributed amongst the $k$ sub-profiles.  That $W_3(\delta)$ is $k$-balanced now follows from the fact that each sub-profile is $1$-balanced.

\section*{Appendix B}

The second author is updating and adding new functionality to the \texttt{qFunctions} package as a part of the Austrian Science Fund (FWF) sponsored project P-34501N ``Partition identities using the weighted word approach" he is running. The most up to date released version of this package as well as the up to date information can be found on the author's website \url{https://akuncu.com/qfunctions/} and on the RISC website  \url{https://risc.jku.at/software/}.

One big difference between earlier work and this paper appears in the description of the profiles for cylindric partitions. This paper follows the profile definitions of Han and Xiong \cite{HX}, however, in \cite{CW, CDU, W} the profiles for cylindric partitions are given in an equivalent but different representation.  One can call \texttt{SwitchCPProfileRepresentation} and provide a profile either in the form used in \cite{CW} or in \cite{HX} to get its equivalent representation. This transformation is simple and, when needed, it is now automatically detected and done by the \texttt{qFunctions} package.

The original $q$-difference relations \eqref{E:CWrecurrence} Corteel and Welsh noted for cylindric partitions \cite{CW} is implemented in \texttt{qFunctions}. The functionality to find all the $q$-difference relations in a coupled system through \eqref{E:CWrecurrencegeneral}-\eqref{E:DSPPrecurrencegeneral} is also added to the package and it is tested meticulously during the development of this work. These functions are called \texttt{WeightedCPFunctionalEquationSystem}, \texttt{WeightedDSPPFunctionalEquationSystem}, and \texttt{WeightedDCPFunctionalEquationSystem} similar to the original \texttt{CylindricFunctionalEquationCreator} with one additional input of the weights at the end. These functions and other related functionalities will be added to the next release of the package. It should be noted that these functionality is not necessary for the proofs of this paper. The implementation is being done solely for assisting future users with problems involving larger coupled systems of equations.

On top of the automated $q$-difference equation system setting functions. The second author also implemented experimental tools where one can choose a product as in the right-hand side of \eqref{E:Borodingeneral} or \eqref{E:DSPPproductgeneral} and try any profile against this product to find weights $a_i$'s which would give the product as the generating function of the selected object. These are done by setting up the linear systems induced by the $W^a_i$ sets presented in the Propositions \ref{P:CPgeneralproduct} and \ref{P:DSPPproductgeneral} against the selected residue classes and modulus. These functionalities are called \texttt{FitCPWeights} and \texttt{FitDSPPWeights}.


\begin{thebibliography}{99}

\bibitem{qFuncs} J. Ablinger and A.K. Uncu, \texttt{qFunctions }\textit{- A Mathematica package for $q$-series and partition theory applications}, J. Symbolic Comput. 107, (2021), 145-166.

\bibitem{A} G.E, Andrews, {\it The Theory of Partitions}, Cambridge University Press (1984).

\bibitem{AB} G.E. Andrews and R.J. Baxter, {\it A motivated proof of the Rogers--Ramanujan identities}, Am. Math. Mon. {\bf 96} (1989),  401--409.

\bibitem{ACW} G.E. Andrews, A. Schilling, S.O. Warnaar, {\it An $A_2$ Bailey lemma and Rogers-Ramanujan-type identities}, J. Amer. Math. Soc., \textbf{12} (3), (1999) 677--702.

\bibitem{AP} G.E. Andrews and P. Paule, {\it MacMahon's partition analysis XIII: Schmidt type partitions and modular forms,} J. Num. Theo. (2021). \url{https://doi.org/10.1016/j.jnt.2021.09.008}

\bibitem{B} A. Borodin, {\it Periodic Schur processes and cylindric partitions}, Duke Math. J. {\bf 140} (2007), 391--468.

\bibitem{CDU} S. Corteel, J. Dousse and A.K. Uncu, {\it Cylindric partitions and some new $A_2$ Rogers--Ramanujan identies,} Proc. Amer. Math. Soc. \textbf{150} (3), (February 2022), 481--497.

\bibitem{CW}  S. Corteel and T. Welsh, {\it The $A_2$ Rogers--Ramanujan identities revisited}, Annals of Comb. {\bf 23} (2019), 683-–694.

\bibitem{FW} B. Feigin, O. Foda, and T. Welsh, {\it Andrews-Gordon type identities from combinations of Virasoro characters}, Ramanujan J. {\bf 17} (2008), no. 1, 33--52.

\bibitem{GM} A.M. Garsia and S.C. Milne, {\it A Rogers--Ramanujan bijection,} J. Comb. Theo. Ser. A {\bf 31} (1981), 289--339.

\bibitem{GK} I.M. Gessel and C. Krattenthaler, {\it Cylindric partitions,} Trans. Amer. Math. Soc. {\bf 349} (1997), 429--479.

\bibitem{G} H. G\"{o}llnitz, {\it Partitionen mit Differenzenbedingungen,} Journal f\"{u}r die reine und angewandte Mathematik {\bf 225} (1967), 154--190.

\bibitem{HX} G.-N. Han and H. Xiong, {\it Skew doubled shifted plane partitions: calculus and asymptotics}, Amer. Inst. Math. Sci. {\bf 29} (2021),  1841--1857.

\bibitem{HoloFuncs} C. Koutschan, \textit{Advanced Applications of the Holonomic Systems Approach}, RISC, Johannes Kepler University, Linz. PhD Thesis. September 2009.

\bibitem{LW}  J. Lepowsky and R.L. Wilson, {\it The Rogers--Ramanujan identities: Lie theoretic interpretation and proof,} Proc. Nat. Acad. Sci. {\bf 78} (1981), 699--701.

\bibitem{Paule} P. Paule, On identities of the Rogers–Ramanujan type, J. Math. Anal. Appl. \textbf{107} (1985), 255--284.

\bibitem{S}  F. Schmidt, {\it Interrupted partitions,} Problem 10629, Am. Math. Mon. {\bf 104} (1999), 87--88.

\bibitem{Sills} A. Sills, {\it An Invitation to the Rogers--Ramanujan Identities}, CRC Press (2017).

\bibitem{U} A.K. Uncu, {\it Weighted Rogers--Ramanujan partitions and Dyson crank}, Ramanujan J. {\bf 46} (2018), 579--591.

\bibitem{W} O. Warnaar, {\it The $A_2$ Andrews--Gordon identities and cylindric partitions,} \texttt{https://arxiv.org/abs/2111.07550}

\end{thebibliography}
\end{document}